\documentclass[a4paper, 12pt]{amsart}
\usepackage[T1]{fontenc}
\usepackage{xspace}
\usepackage{enumerate}
\usepackage[utf8]{inputenc} 
\usepackage[english]{babel}
\usepackage{lmodern}
\usepackage{microtype}
\usepackage{amssymb}
\usepackage{verbatim}
\usepackage{amsthm}
\usepackage{amsmath}
\usepackage{mathtools}
\usepackage{color}
\usepackage{tikz}
\usepackage{threeparttable}
\usepackage{graphicx}
\usepackage{caption}
\usepackage{paralist}
\usepackage[pagebackref=true, pdfborder={0 0 0}, linktocpage]{hyperref}
\usepackage{backref}
\usepackage{aliascnt}
\usepackage{cleveref}
\usepackage{dsfont}
\usepackage{mathrsfs}
\usepackage{bm}

\newtheorem{theorem}{Theorem}[section]
\newtheorem{lemma}[theorem]{Lemma}

\theoremstyle{remark}
\newtheorem{remark}[theorem]{Remark}

\newtheorem{definition}[theorem]{Definition}

\numberwithin{equation}{section}

\begin{document}

\title[Regularization by noise for averaged NSE]{Regularization by noise of an averaged version of the Navier-Stokes equations}
\author{Theresa Lange}
\address{Fakult\"at f\"ur Mathematik, Universit\"at Bielefeld, D-33501 Bielefeld, Germany}
\email{tlange@math.uni-bielefeld.de}

\date{May 30, 2022}

\begin{abstract}
	In Tao 2016, the author constructs an averaged version of the deterministic three-dimensional Navier-Stokes equations (3D NSE) which experiences blow-up in finite time. In the last decades, various works have studied suitable perturbations of ill-behaved deterministic PDEs in order to prevent or delay such behavior. A promising example is given by a particular choice of stochastic transport noise closely studied in Flandoli et al. 2021. We analyze the model in Tao 2016 in view of these results and discuss the regularization skills of this noise in the context of the averaged 3D NSE.
\end{abstract}


\maketitle 

\tableofcontents

\section{Introduction}\label{sec:intro}
\noindent
Consider the Navier-Stokes equations (NSE) on $\mathbb{R}^3$ describing the dynamics of an incompressible viscous fluid
\begin{equation}\label{eq:NSE}
	\begin{aligned}
		\partial_t u + (u \cdot \nabla)u + \nabla p &= \Delta u,\\
		{\rm div} u &= 0,\\
		u(0, \cdot) &= u_0,
	\end{aligned}
\end{equation}
with vector-valued velocity field $u:[0,\infty) \times \mathbb{R}^3 \rightarrow \mathbb{R}^3$ and scalar-valued pressure field $p:[0,\infty)\times \mathbb{R}^3 \rightarrow \mathbb{R}$. By a rescaling argument, we may assume unitary viscosity. The global well-posedness of 3D NSE forms a long-standing open problem and has attracted the attention of researchers ever since: it was established in the regime of small initial data (see e.g. \cite{koch2001}) as well as in the hyperdissipative case (replacing $\Delta$ by $-(-\Delta)^{\alpha}$, for $\alpha \geq \frac{5}{4}$ see e.g. \cite{katz2002}). The general case remains unresolved; in particular, there exists evidence for blow-up of solutions to variants of \eqref{eq:NSE} as exploited in e.g. \cite{const1986}, \cite{katz2005}, \cite{koch2001}, \cite{lei2009}, \cite{li2008}, \cite{mont2001}, \cite{ohkitani2017}, \cite{plech2003}. This note shall focus on the following result presented in \cite{tao2016}: consider the projection of \eqref{eq:NSE} onto the space of divergence-free vector fields, reading
\begin{equation}\label{eq:projNSE}
	\begin{aligned}
		\partial_t u &= \Delta u +B(u,u),\\
		u(0,\cdot)&=u_0.
	\end{aligned}
\end{equation}
Here $B$ denotes the bilinear Euler operator (see e.g. in \cite{tao2016}) which is symmetric and for sufficiently regular $u$ satisfies the cancellation property
\begin{equation}\label{eq:cancel}
	\langle B(u,u), u \rangle_{L^2(\mathbb{R}^3)} = 0
\end{equation}
where $\langle u, v \rangle_{L^2(\mathbb{R}^3)} := \int_{\mathbb{R}^3} u(x) \cdot v(x) {\rm d} x$ in $L^2(\mathbb{R}^3)$. Standard harmonic analysis approaches then use both symmetry and \eqref{eq:cancel} to investigate the behavior of the system's energy in view of well-posedness. However, \cite{tao2016} demonstrates that such an ansatz will not be successful: the author constructs an explicit, "averaged" version of \eqref{eq:projNSE} in which symmetry and cancellation property remain valid, and for which the corresponding system experiences a blow-up in finite time.\\
In recent years it has been investigated whether for ill-posed deterministic PDEs it is possible to construct a perturbation yielding higher regularity of the corresponding perturbed system. Especially in the case of stochastic perturbations the motive behind such constructions is that noise may have a smoothing effect. This goes under the name of regularization by noise and has been exploited in various settings; as an overview see e.g. \cite{flandoli2011} and \cite{gess2016}, or in the context of NSE e.g. \cite{bianchi2020}. A particularly interesting stochastic perturbation is given by a specific type of transport noise which in the context of the vorticity formulation of \eqref{eq:NSE} and similar, more general models proved to yield existence of solutions for arbitrarily long time with large probability, see \cite{flandoli2021a} and \cite{flandoli2021b}.\\
The averaged NSE from \cite{tao2016} form an excellent candidate to test the regularization effect of this noise: though it explodes in finite time, we want to exploit whether at least locally in time the model provides enough regularity in order for noise to delay the blow up. In \cite{flandoli2021b}, this can be achieved if the nonlinearity of the system admits for a continuity, growth and local monotonicity condition. It turns out that similar to the case of the standard NSE \eqref{eq:NSE}, we do not obtain local existence and uniqueness of solutions to the averaged NSE on the velocity level in the sense that these three conditions cannot be shown to hold true. In turn, we give a lower bound on the order of derivatives which allow for an analysis as in \cite{flandoli2021b}.\\
This note is therefore structured as follows: in Section \ref{sec:mainIngredients} we repeat the core results considered in this note, namely the averaged NSE from \cite{tao2016} as well as the regularization results from \cite{flandoli2021a} and \cite{flandoli2021b}. In particular we will formulate the result from \cite{flandoli2021b} in the case of divergence-free vector fields with the proof given in Appendix \ref{app:proofThmDivfree}. Note that regularization was shown to hold for systems on the torus, hence in Section \ref{sec:mainResults} we first show that the blow-up result in \cite{tao2016} carries over to the periodized averaged NSE (see Section \ref{subsec:periodTao}). The rest of the section then gives the violation results of the above three conditions as well as the lower bound on regularizable derivatives (see Section \ref{subsec:localW}).

\subsection{Notation}\label{subsec:notation}
Let $L^2(\mathbb{K}^d)$ denote the space of square-integrable functions on $\mathbb{K}^d$ with norm $\|\cdot\|_{L^2(\mathbb{K}^d)}$ where $\mathbb{K}^d = \mathbb{R}^d$ or $\mathbb{K}^d = \mathbb{T}^d := \mathbb{R}^d / \mathbb{Z}^d$. Further let $H^{\alpha}(\mathbb{K}^d), \alpha \in \mathbb{R},$ denote the Sobolev space endowed with norm $\|\cdot\|_{H^{\alpha}(\mathbb{K}^d)} := \left\|({\rm Id} -\Delta)^{\frac{\alpha}{2}}\cdot\right\|_{L^2(\mathbb{K}^d)}$. As already used in the introduction, we will further use $\langle \cdot, \cdot \rangle_{L^2(\mathbb{K}^d)}$ to denote the inner product in $L^2(\mathbb{K}^d)$, and $\langle \cdot , \cdot \rangle$ for the dual pairing of $H^{\alpha}(\mathbb{K}^d)$ and $H^{-\alpha}(\mathbb{K}^d)$. Furthermore let $H^{\alpha}_{\rm df}(\mathbb{K}^d)$ denote the space of divergence-free functions in $H^{\alpha}(\mathbb{K}^d)$. On the other hand, denoting $L^p(0,T;Z)$ the set of all $u:[0,T]\rightarrow Z$ in $L^p$ for some linear space $Z$ with norm $\|\cdot\|_Z$, we define the fractional Sobolev spaces by 
\[W^{\beta, p}(0,T;Z):= \left\{ u \in L^p(0,T;Z): \int_0^T \int_0^T \frac{\|u(t)-u(s)\|_Z^p}{|t-s|^{1+\beta p}}{\rm d}t {\rm d}s < \infty\right\}\]
for $\beta \in (0,1), p >1$. If it exists, we will denote by $\mathcal{F}_{\mathbb{K}^d}f$ the Fourier transform of a function $f$ on $\mathbb{K}^d$. Whenever it is clear from the context, we drop the $\mathbb{K}^d$ in the notation. Finally, let ${\rm supp}$ denote the support of a function, and let $x \lesssim y$ denote $x \leq C y$ for some constant $C>0$.

\subsection{Acknowledgements}
The author would like to express her deepest gratitude to Martina Hofmanov\'a for many fruitful discussions. This project has received funding from the European Research Council (ERC) under the European Union's Horizon 2020 research and innovation programme (grant agreement No. 949981).

\section{Main ingredients}\label{sec:mainIngredients}
\noindent
This section shall serve the purpose of bringing together the various objects considered in this work. It will mainly consist of repetitions of external work and will set the notation throughout.

\subsection{The averaged NSE}\label{subsec:avNSE}
The main idea in \cite{tao2016} is to construct a modification of the projected NSE \eqref{eq:projNSE} which on the one hand preserves energy as well as symmetry and cancellation property of the nonlinearity, and on the other hand experiences a blow-up. In order to do so, consider a suitable frequency decomposition of the projected NSE and "average out", i.e. eliminate a carefully selected choice of frequencies, resulting in a system of the following form
\begin{equation}\label{eq:avNSE}
	\begin{aligned}
		\partial_t u &= \Delta u +C(u,u),\\
		u(0,\cdot) &= u_0,
	\end{aligned}
\end{equation}
where $C$ is called a \textit{local cascade operator}, introduced below. Thus allowing only for localized frequency interactions gives rise to a concrete blow-up mechanism. Considering the subcritical case of mild solutions in the regularity class $H_{\rm df}^{10}(\mathbb{R}^3)$, the main result in \cite{tao2016} is as follows:
\begin{theorem}[cf. {\cite[Theorem 3.3]{tao2016}} ]\label{thm:blowupTao}
There exist a symmetric local cascade operator $C$ satisfying the cancellation property, and a divergence-free vector field $u_0$ such that there does not exist any global mild solution $u:[0,\infty)\rightarrow H_{\rm df}^{10}(\mathbb{R}^3)$ to \eqref{eq:avNSE}.
\end{theorem}
\noindent
The construction of such an operator $C$ in \cite{tao2016} is inspired by the work of \cite{katz2005} in case of the dyadic hyperdissipative NSE: heuristically, a solution $u$ to the projected NSE \eqref{eq:projNSE} can be approximated by a wavelet decomposition of the form
\begin{equation}\label{eq:waveletDecomp}
\sum_n u_n (t)\psi_n(x)
\end{equation}
for a suitable orthonormal basis $\{\psi_n\}$ in $L^2(\mathbb{R}^3)$, and the wavelet coefficients $u_n$ evolve according to the following ODEs
\begin{equation}\label{eq:dyadicX}
\partial_t u_n = 2^{\frac{5n}{2}} u_{n-1}^2 -2^{2 n}u_n - 2^{\frac{5(n+1)}{2}} u_n u_{n+1}.
\end{equation}
\noindent
Hence the corresponding energy equation reads
\begin{equation*}
\partial_t \left(\frac{1}{2} u_n^2\right) = 2^{\frac{5n}{2}} u_{n-1}^2u_n -2^{2n}X_n^2 - 2^{\frac{5(n+1)}{2}} u_n^2 u_{n+1},
\end{equation*}
\noindent
encoding a 'low-to-high-frequency-cascade': the energy from the previous scale $n-1$ enters scale $n$ and, apart from some dissipated portion, will be completely transported to the next scale $n+1$. This system experiences a blow-up in $H^{2+\epsilon}(\mathbb{R}^3)$ for small $\epsilon >0$ and dissipation exponent $\alpha <\frac{1}{4}$; however, in \cite{barbato2011} it has been shown that in the dissipation range containing the standard NSE, the corresponding model dissipates energy fast enough to prevent such a blow-up. The construction in \cite{tao2016} allows for a decomposition \eqref{eq:waveletDecomp} in such a way that the system of coefficients $u_n$ captures an additional time delay in which energy first accumulates at one scale and is then abruptly transported to the next. This way the energy cascade outruns the dissipation and yields a blow-up in finite time. Let us now give the precise formulation of an operator $C$ enabling such behaviour:
\begin{definition}[cf. {\cite[Section 4]{tao2016}}]\label{def:cascadeOp}
Let $\epsilon_0 \in (0,1)$ and $m\in \mathbb{N}$. Furthermore let $B_1,...,B_m$ be balls in the annulus $\{\xi \in \mathbb{R}^3: 1 < |\xi|<1+\frac{\epsilon_0}{2}\}$ such that $B_1,...,B_m,-B_1,...,-B_m$ are disjoint. For $n\in \mathbb{Z}$ and $i\in\{1,...,m\}$, let $\psi_{i,n}:\mathbb{R}^3\rightarrow \mathbb{R}^3$ be rescaled $L^2$-functions with
\[\psi_{i,n}(\cdot) := (1+\epsilon_0)^{\frac{3n}{2}}\psi_i((1+\epsilon_0)^n \cdot)\]
where $\psi_i \in H_{\rm df}^{10}(\mathbb{R}^3)$ are Schwartz functions with Fourier transform supported on $B_i \cup -B_i$ and normalized to $\|\psi_i\|_{L^2}=1$. Let $S=\{(0,0,0),(1,0,0),(0,1,0),(0,0,1)\}$ and $\alpha_{i_1,i_2,i_3,\mu_1,\mu_2,\mu_3}\in \mathbb{R}$ be bounded structure constants where $i_1,i_2,i_3 \in \{1,...,m\}$ and $(\mu_1,\mu_2,\mu_3)\in S$. Then a \emph{local cascade operator} $C:H_{\rm df}^{10}(\mathbb{R}^3) \times H_{\rm df}^{10}(\mathbb{R}^3) \rightarrow H_{\rm df}^{-10}(\mathbb{R}^3)$ is defined by
\begin{equation}
\begin{aligned}
C(u,v) := &\sum_{n \in \mathbb{Z}}\sum_{(i_1,i_2,i_3,\mu_1,\mu_2,\mu_3)\in\{1,...,m\}^3\times S}\alpha_{i_1,i_2,i_3,\mu_1,\mu_2,\mu_3} (1+\epsilon_0)^{\frac{5n}{2}}\\
&\qquad \langle u, \psi_{i_1,n+\mu_1}\rangle_{L^2} \langle v, \psi_{i_2,n+\mu_2} \rangle_{L^2} \psi_{i_3,n+\mu_3}
\end{aligned}
\end{equation}
for $u,v \in H_{\rm df}^{10}(\mathbb{R}^3)$.
\end{definition}
\begin{remark}
	In the following we will rather use the short hand notation $\sum_{n,i,\mu}$ as well as $\alpha_{i,\mu}$.
\end{remark}
\noindent
As done in \cite{tao2016}, requiring the symmetry condition
\[\alpha_{i_1,i_2,i_3,\mu_1,\mu_2,\mu_3} = \alpha_{i_2,i_1,i_3,\mu_2,\mu_1,\mu_3}\]
as well as the cancellation condition
\[\sum_{\{a,b,c\}=\{1,2,3\}} \alpha_{i_a,i_b,i_c,\mu_a,\mu_b,\mu_c} =0\]
for all $i_1,i_2,i_3 \in \{1,...,m\}$ and $(\mu_1,\mu_2,\mu_3)\in S$, this ensures that $C$ is symmetric and satisfies the cancellation property
\[\langle C(u,u),u\rangle_{L^2}=0\quad \forall u \in H_{\rm df}^{10}(\mathbb{R}^3).\]
Next consider the corresponding Cauchy problem
\begin{equation}
	\begin{aligned}
		\partial_t u &= \Delta u +C(u,u),\\
		u(0,\cdot) &= u_0 := \psi_{1,n_0},
	\end{aligned}
\end{equation}
for some $n_0 \in \mathbb{N}$ sufficiently large, and assume that there exists a mild solution $u:[0,\infty) \rightarrow H_{\rm df}^{10}(\mathbb{R}^3)$. Then the following holds:
\begin{lemma}[cf. {\cite[Lemma 4.1]{tao2016}}]\label{lem:Tao}
	For each $n\in\mathbb{Z}, t \geq 0$ and $i\in\{1,...,m\}$ define
	\begin{align*}
		\left(\mathcal{F}_{\mathbb{R}^3}u_{i,n}(t)\right)(\xi) &:= \left(\mathcal{F}_{\mathbb{R}^3}u(t)\right)(\xi)\mathds{1}_{\{ \xi \in (1+\epsilon_0)^n(B_i \cup -B_i)\}},\\
		X_{i,n}(t) &:= \langle u(t), \psi_{i,n} \rangle_{L^2} = \langle u_{i,n}(t),\psi_{i,n}\rangle_{L^2},\\
		E_{i,n} &:= \frac{1}{2} \|u_{i,n}(t)\|_{L^2}^2,
	\end{align*}
	then
	\begin{enumerate}
		\item it holds
		\begin{equation}\label{eq:boundXin}
			\sup_{t\in[0,T]}\sup_{n\in\mathbb{Z}}\sup_{i\in\{1,...,m\}} \left(1+(1+\epsilon_0)^{10n}\right)|X_{i,n}(t)|<\infty
		\end{equation}
		and
		\begin{equation}
			\sup_{t\in[0,T]}\sup_{n\in\mathbb{Z}}\sup_{i\in\{1,...,m\}} \left(1+(1+\epsilon_0)^{10n}\right)|E_{i,n}(t)|<\infty
		\end{equation}
		for all $T \in(0,\infty)$,
		\item for any $n \in \mathbb{Z}, i\in\{1,...,m\}$
		\begin{align}
			E_{i,n}(0) &= \frac{1}{2}X_{i,n}(0)^2,\\
			X_{i,n}(0) &= \mathds{1}_{\{(i,n)=(1,n_0)\}},
		\end{align}
		\item for any $n \in \mathbb{Z}, i\in\{1,...,m\}$
		\begin{equation}\label{eq:evolXin}
			\begin{aligned}
				\partial_t X_{i,n} &= \sum_{i_1,i_2\in\{1,...,m\}}\sum_{\mu\in S} \alpha_{i_1,i_2,i,\mu} (1+\epsilon_0)^{\frac{5(n-\mu_3)}{2}} X_{i_1,n-\mu_3+\mu_1}X_{i_2,n-\mu_3+\mu_2}\\
				&\hspace{0,5cm} + O\left((1+\epsilon_0)^{2n}\sqrt{E_{i,n}}\right)
			\end{aligned}
		\end{equation}
		and
		\begin{equation}
			\partial_t E_{i,n} \leq \sum_{i_1,i_2\in\{1,...,m\}}\sum_{\mu \in S} \alpha_{i_1,i_2,i,\mu} (1+\epsilon_0)^{\frac{5(n-\mu_3)}{2}} X_{i_1,n-\mu_3+\mu_1}X_{i_2,n-\mu_3+\mu_2}X_{i,n},
		\end{equation}
		\item for any $n \in \mathbb{Z}, i\in\{1,...,m\}$
		\begin{equation}
			\frac{1}{2}X_{i,n}(t)^2 \leq E_{i,n}(t) \leq \frac{1}{2}X_{i,n}(t)^2 + O\left((1+\epsilon_0)^{2n}\int_0^t E_{i,n}(s){\rm d}s\right),
		\end{equation}
		\item it holds
		\begin{equation}\label{eq:initCond}
			X_{i,n}(t)=0=E_{i,n}(t)
		\end{equation}
		for all $n<n_0, i\in\{1,...,m\}, t\geq 0$.
	\end{enumerate}
\end{lemma}
\noindent
Observe that up to the $O$-terms, \eqref{eq:evolXin} is of the form \eqref{eq:dyadicX} and existence of a global mild solution $u$ implies boundedness of the $X_{i,n}$ as formalized in \eqref{eq:boundXin}. Thus in order to prove Theorem \ref{thm:blowupTao}, the author constructs a sequence $(X_{i,n})$ violating Lemma \ref{lem:Tao} which specifies a blow-up in $H_{\rm df}^{10}$ (cf. \cite[Theorem 4.2]{tao2016} as well as the construction in \cite[Section 6]{tao2016}).

\subsection{Regularization by transport noise}\label{subsec:noise}
In this section, we shall introduce the specific choice of transport noise (cf. \cite[Section 2]{flandoli2021a} and \cite[Section 1.2]{flandoli2021b}): on the $d$-dimensional torus $\mathbb{T}^d$ consider the following noise
\begin{equation}\label{eq:noise}
\frac{\sqrt{C_d \nu}}{\|\theta\|_{\ell^2}}\sum_{k \in \mathbb{Z}_0^d}\sum_{i=1}^{d-1}\theta_k \Pi((\sigma_{k,i}\cdot\nabla)\cdot)W^{k,i}.
\end{equation}
Here $C_d =\frac{d}{d-1}, d \geq 2, \nu >0$, $\Pi$ is the Leray projection and the individual components are as follows: let $\ell^2 = \ell^2(\mathbb{Z}_0^d)$ denote the space of square-summable sequences indexed by $\mathbb{Z}_0^d=\mathbb{Z}^d\setminus\{0\}$ and choose a sequence $\theta=(\theta_k)_{k \in \mathbb{Z}_0^d} \in \ell^2$ with finitely many non-zero components such that $\theta$ satisfies a symmetry condition
\[ \theta_k = \theta_l \quad \forall k, l \in \mathbb{Z}_0^d,\quad |k|=|l|.\]
Further, let $\{ \sigma_{k,i} : k \in \mathbb{Z}_0^d, i=1,...,d-1\}$ be periodic divergence-free smooth vector fields forming a complex orthonormal system of the space
\begin{equation}\label{eq:HC}
H_{\mathbb{C}}=\left\{v \in L^2(\mathbb{T}^d,\mathbb{C}^d):\quad \int_{\mathbb{T}^d} v {\rm d} x = 0, {\rm div} v =0\right\}
\end{equation}
and which are defined as follows:
\[\sigma_{k,i}(x) = a_{k,i}e^{2\pi{\rm i}k\cdot x}, \quad x \in \mathbb{T}^d, k\in\mathbb{Z}_0^d, i=1,...,d-1.\]
Here ${\rm i}$ denotes the imaginary unit and considering a partition $\mathbb{Z}_{+}^d, \mathbb{Z}_{-}^d$ of $\mathbb{Z}_0^d$ such that $\mathbb{Z}_0^d = \mathbb{Z}_{+}^d \cup \mathbb{Z}_{-}^d, \mathbb{Z}_{+}^d = - \mathbb{Z}_{-}^d$, choose for any $k \in \mathbb{Z}_{+}^d$ the set $\{a_{k,i}:i=1,...,d-1\}$ to be an ONB of $k^{\perp} := \{y \in \mathbb{R}^d: y \cdot k =0\}$, and define $a_{k,i} = a_{-k,i}$ for any $k \in \mathbb{Z}_{-}^d$.\\
Finally let $\{W^{k,i}:k \in \mathbb{Z}_0^d,i=1,...,d-1\}$ be a family of complex Brownian motions on a probability space $(\Omega, \mathcal{F},\mathbb{P})$ such that
\begin{equation}\label{eq:Wcomplconj}
	\overline{W^{k,i}} = W^{-k,i}
 \end{equation}
and their cross-variation satisfies
\begin{equation}\label{eq:Wcov}
	\left[W^{k,i},W^{l,j}\right]_t = 2t \delta_{k+l}\delta_{i-j} \quad \forall k,l \in \mathbb{Z}_0^d, i,j \in \{1,...,d-1\}
\end{equation}
in order for $W^{k,i}$ and $W^{l,j}$ to be independent whenever $k \neq \pm l$ and $i \neq j$.\\
\textbf{Example}: In \cite{flandoli2021a}, the authors consider a family $\{B^{k,i}:k \in \mathbb{Z}_0^d,i=1,...,d-1\}$ of standard real-valued Brownian motions and define for $k \in \mathbb{Z}_{+}^d$
\[ W^{k,i} := B^{k,i} +{\rm i}B^{-k,i}\]
and for $k \in \mathbb{Z}_{-}^d$
\[ W^{k,i} := B^{k,i} -{\rm i}B^{-k,i}.\]
It is easy to check that $\{W^{k,i}:k \in \mathbb{Z}_0^d, i=1,...,d-1\}$ then satisfy \eqref{eq:Wcomplconj} and \eqref{eq:Wcov}.

\subsubsection{The vorticity formulation of NSE}\label{subsubsec:vortNSE}
For $d=3$, the vorticity $\xi := \nabla \times u$ of the standard NSE \eqref{eq:NSE} evolves according to
\begin{equation}
\partial_t \xi +\mathcal{L}_u \xi =\Delta \xi
\end{equation}
with Lie derivative $\mathcal{L}_u \xi = (u\cdot \nabla)\xi - (\xi \cdot \nabla)u$ consisting of a transport and a vortex stretching term, respectively. As discussed in \cite{flandoli2021a}, we may heuristically recover the form of noise \eqref{eq:noise} here when separating the vorticity into large-scale and small-scale component and treating the later as a random perturbation of the former. The small-scale vortex stretching term, however, complicates the regularization-by-noise analysis but it is shown in \cite{flandoli2021a} that the transport term on its own already has sufficient regularization skills. More precisely, let $B_H(R_0)$ denote the ball of radius $R_0$ in the real subspace $H$ of $H_{\mathbb{C}}$, then the authors of \cite{flandoli2021a} are able to show the following result:
\begin{theorem}[cf. {\cite[Corollary 1.5]{flandoli2021a}}]
For $R_0 >0$, $T>0$, and $\epsilon >0$, there exists $\theta \in \ell^2$ such that for all $\xi_0 \in B_H(R_0)$
\begin{equation}\label{eq:regFL}
	{\rm d} \xi +\mathcal{L}_u \xi {\rm d}t = \Delta \xi {\rm d}t + \frac{\sqrt{C_3 \nu}}{\|\theta\|_{\ell^2}}\sum_{k \in \mathbb{Z}_0^3}\sum_{i=1}^{3}\theta_k \Pi((\sigma_{k,i}\cdot\nabla)\xi)\circ {\rm d}W^{k,i}
\end{equation}
admits a unique strong solution up to time $T$ with probability no less than $1-\epsilon$.
\end{theorem}
\noindent
For the proof rewrite the Stratonovich equation \eqref{eq:regFL} into its corresponding It\^{o}-formulation which by the above choice of parameters is of the form
\begin{equation}\label{eq:FLvortIto}
	{\rm d} \xi +\mathcal{L}_u \xi {\rm d}t = \left(\Delta \xi + S_{\theta}(\xi)\right){\rm d}t + \frac{\sqrt{C_3 \nu}}{\|\theta\|_{\ell^2}}\sum_{k \in \mathbb{Z}_0^3}\sum_{i=1}^{3}\theta_k \Pi((\sigma_{k,i}\cdot\nabla)\xi) {\rm d}W^{k,i}
\end{equation}
with It\^{o}-Stratonovich correction denoted by $S_{\theta}(\xi)$. Then they show that there exists a suitable choice of sequence $(\theta^N)_{N \in \mathbb{N}}$ such that in a suitable sense specified in \cite{flandoli2021a}, in the limit of $N\rightarrow \infty$ the martingale part in \eqref{eq:FLvortIto} vanishes and
\[ \lim_{N \rightarrow \infty} S_{\theta^N}(\xi) = \frac{3}{5}\nu\Delta \xi.\]
Hence obtain the limiting equation
\begin{equation}\label{eq:FLlimit}
	\partial_t \xi + \mathcal{L}_u \xi = \left(1+\frac{3}{5}\nu\right)\Delta \xi
\end{equation}
and the claim then follows by using existence of a unique global strong solution to \eqref{eq:FLlimit} for large enough $\nu$.

\subsubsection{Criteria for delayed blow-up}\label{subsubsec:criteriaBlowUp}
In the case of $\mathbb{T}^d$, $d \geq 2$, consider systems of more general form, namely
\begin{equation}\label{eq:generalDet}
\begin{aligned}
	\partial_t u &= -(-\Delta)^{\alpha} u + F(u),\\
	u(0,\cdot) &= u_0,
\end{aligned}
\end{equation}
for $\alpha \geq 1$ and a fixed initial condition $u_0 \in  L^2(\mathbb{T}^d)$. Regularization by transport noise is obtained under the following structural assumptions:
\begin{itemize}
	\item[(H1)] \textit{Continuity}: There exist $\beta_1\geq 0$ and $\eta\in(0,\alpha)$ such that $F:H^{\alpha - \eta} (\mathbb{T}^d)\rightarrow H^{-\alpha}(\mathbb{T}^d)$ is continuous and
	\begin{equation*}
		\|F(u)\|_{H^{-\alpha}(\mathbb{T}^d)} \lesssim \left(1+\|u\|_{L^2(\mathbb{T}^d)}^{\beta_1}\right)\left(1+\|u\|_{H^{\alpha}(\mathbb{T}^d)}\right).
	\end{equation*}
	\item[(H2)] \textit{Growth}: There exist $\beta_2 \geq 0$ and $\gamma_2 \in (0,2)$ such that
	\begin{equation*}
		|\langle F(u),u\rangle| \lesssim \left(1+\|u\|_{L^2(\mathbb{T}^d)}^{\beta_2}\right)\left(1+\|u\|_{H^{\alpha}(\mathbb{T}^d)}^{\gamma_2}\right).
	\end{equation*}
	\item[(H3)] \textit{Local monotonicity}: There exist $\beta_3 , \kappa \geq 0$, $\gamma_3 \in (0,2)$ such that $\beta_3 + \gamma_3 \geq 2$, $\kappa + \gamma_3 \leq 2$ and
	\begin{equation*}
		\begin{aligned}
			&|\langle u-v, F(u) - F(v) \rangle | \\
			&\lesssim \|u-v\|_{L^2(\mathbb{T}^d)}^{\beta_3} \|u-v\|_{H^{\alpha}(\mathbb{T}^d)}^{\gamma_3}\left(1+\|u\|_{H^{\alpha}(\mathbb{T}^d)}^{\kappa} + \|v\|_{H^{\alpha}(\mathbb{T}^d)}^{\kappa}\right).
		\end{aligned}
	\end{equation*}
	\item[(H4)] \textit{Admissible initial conditions}: There exists $\mathcal{K} \subset L^2(\mathbb{T}^d)$ convex, closed and bounded with the following property: for any $T>0$, we can find $\nu >0$ big enough such that the deterministic Cauchy problem
	\begin{equation}\label{eq:cauchyFGL}
	\begin{aligned}
		\partial_t u&= -(-\Delta)^{\alpha}u + \nu\Delta u +F(u),\\
		u(0,\cdot)&=u_0,
	\end{aligned}
	\end{equation}
	admits a global solution $u:=u(\cdot; u_0, \nu) \in L^2(0,T; H^{\alpha}(\mathbb{T}^d))\cap C([0,T]; L^2(\mathbb{T}^d))$ for any $u_0 \in \mathcal{K}$, and moreover
	\begin{equation}
		\sup_{u_0 \in \mathcal{K}} \sup_{t \in [0,T]} \|u(t; u_0, \nu)\|_{L^2(\mathbb{T}^d)} < \infty.
	\end{equation}
\end{itemize}
Given a deterministic $u_0 \in L^2(\mathbb{T}^d)$, let $\tau(u_0, \nu, \theta)$ denote the random maximal time of existence of solutions $u(t; u_0, \nu, \theta)$ to 
\begin{equation}\label{eq:regFGL21}
	\begin{aligned}
		{\rm d} u &= (-(-\Delta)^{\alpha} u + F(u)){\rm d} t + \frac{\sqrt{C_d \nu}}{\|\theta\|_{\ell^2}}\sum_{k \in \mathbb{Z}_0^d}\sum_{i=1}^{d-1}\theta_k (\sigma_{k,i}\cdot\nabla)u \circ {\rm d}W^{k,i},\\
		u(0,\cdot) &= u_0,
	\end{aligned}
\end{equation}
with trajectories in $C([0,T];L^2(\mathbb{T}^d))$. Then
\begin{theorem}[cf. {\cite[Theorem 1.4]{flandoli2021b}}]\label{thm:FGL21}
	Assume $F$ satisfies (H1)-(H3) and $\mathcal{K} \subset L^2(\mathbb{T}^d)$ satisfies (H4). Then for arbitrary large time $T\in(0,\infty)$, $\nu = \nu(T) >0$ as in (H4) and arbitrary small $\epsilon >0$, there exists $\theta \in \ell^2$ such that
	\begin{equation}
		\mathbb{P}\left[\tau(u_0, \nu, \theta)\geq T\right] > 1- \epsilon \quad \forall u_0 \in \mathcal{K}.
	\end{equation}
\end{theorem}
\begin{remark}\label{rem:hypFGL21}
\begin{enumerate}
	\item Assuming exponential decay of the $L^2(\mathbb{T}^3)$-norm of the solution to the deterministic system \eqref{eq:cauchyFGL} as well as existence of a pathwise unique global solution to \eqref{eq:regFGL21} for small initial conditions, then Theorem \ref{thm:FGL21} may even be extended to hold for infinite time horizon (cf. \cite[Theorem 1.4]{flandoli2021b}).
	\item By \cite[Remark 1.3, (iii)]{flandoli2021b}, if $F$ preserves the space of mean-zero functions, then considering the dynamics restricted to this closed subspace of $L^2(\mathbb{T}^d)$ as well as for any fixed constant $R\geq 0$, a suitable choice for $\mathcal{K}$ is
	\begin{equation}\label{eq:H4K}
		\mathcal{K} = \left\{ f \in L^2(\mathbb{T}^d): \int_{\mathbb{T}^d} f {\rm d} x = 0, \|f\|_{L^2} \leq R \right\}.
	\end{equation}
	\item A useful implication of (H2) is given in \cite[Remark 3.5]{flandoli2021b}: for sufficiently small parameter $\delta >0$, standard interpolation yields
	\subitem{(H2')} There exist $\tilde{\beta}_2>0$ and $\tilde{\gamma}_2 <2$ such that
	\begin{equation}
		|\langle F(u), u\rangle| \lesssim \left(1+\|u\|_{H^{\alpha}(\mathbb{T}^d)}^{\tilde{\gamma}_2}\right)\left(1+\|u\|_{H^{-\delta}(\mathbb{T}^d)}^{\tilde{\beta}_2}\right).
	\end{equation}
	\item By \cite[Remark 1.3, (ii)]{flandoli2021b}, hypothesis (H3) can be further generalized to
	\subitem{(H3')} There exist $N \in \mathbb{N}$ and non-negative parameters $\beta_3^j, \gamma_3^j, \kappa_j, \kappa_j'$, $j=1,...,N$ such that $\gamma_3^j \in (0,2), \beta_3^j + \gamma_3^j \geq 2, \gamma_3^j + \kappa_j \leq 2$ for all $j$ and
	\begin{equation}
		\begin{aligned}
			&|\langle u-v, F(u) - F(v) \rangle|\\
			& \lesssim \sum_{j=1}^N \|u-v\|_{L^2(\mathbb{T}^d)}^{\beta_3^j}\|u-v\|_{H^{\alpha}(\mathbb{T}^d)}^{\gamma_3^j}\\
			&\qquad\left(1+\|u\|_{L^2(\mathbb{T}^d)}^{\kappa_j'}+\|v\|_{L^2(\mathbb{T}^d)}^{\kappa_j'}\right)\left(1+\|u\|_{H^{\alpha}(\mathbb{T}^d)}^{\kappa_j}+\|v\|_{H^{\alpha}(\mathbb{T}^d)}^{\kappa_j}\right).
		\end{aligned}
	\end{equation}
\end{enumerate}
\end{remark}
\noindent
Observe that solutions to \eqref{eq:generalDet} need not be divergence-free, hence the noise in \eqref{eq:regFGL21} does not contain the Leray projection $\Pi$ (compare with \eqref{eq:noise}). In the course of this note, we will, however, be in the setting of divergence-free systems. Thus let 
\[ \mathcal{D} := \left\{ u \in L^2(\mathbb{T}^d):\quad {\rm div} u =0\right\}\]
and $\tilde{\tau}(u_0, \nu, \theta)$ denote the analogon to $\tau(u_0, \nu, \theta)$ for 
\begin{equation}\label{eq:regFGL21divfree}
	\begin{aligned}
		{\rm d} u &= (-(-\Delta)^{\alpha} u + F(u)){\rm d} t + \frac{\sqrt{C_d \nu}}{\|\theta\|_{\ell^2}}\sum_{k \in \mathbb{Z}_0^d}\sum_{i=1}^{d-1}\theta_k \Pi((\sigma_{k,i}\cdot\nabla)u) \circ {\rm d}W^{k,i},\\
		u(0,\cdot) &= u_0.
	\end{aligned}
\end{equation}
Then using the tools of \cite{flandoli2021a} in the proof of Theorem \ref{thm:FGL21} gives the following adapted result:
\begin{theorem}\label{thm:FGL21divfree}
	Additionally to the assumptions in Theorem \ref{thm:FGL21}, let $F$ preserve $\mathcal{D}$. Then for arbitrary large time $T>0$ and arbitrary small $\epsilon >0$, there exists $\theta \in \ell^2$ such that
	\begin{equation}
		\mathbb{P}\left[\tilde{\tau}(u_0, \nu, \theta)\geq T\right] > 1- \epsilon \quad \forall u_0 \in \mathcal{K}\cap\mathcal{D}.
	\end{equation}
\end{theorem}
\noindent
The proof shall be given in Appendix \ref{app:proofThmDivfree}.

\section{Main results}\label{sec:mainResults}
\noindent
In this section, we shall bring together the components introduced in Section \ref{sec:mainIngredients}. Since the analysis for the transport noise in Section \ref{subsubsec:criteriaBlowUp} currently works only on the torus, we shall first check whether the analysis from \cite{tao2016} can be transferred to $\mathbb{T}^3$.
\subsection{The averaged NSE on $\mathbb{T}^3$}\label{subsec:periodTao}
Consider the following periodization of the functions $\psi_{i,n}$ in Definition \ref{def:cascadeOp}:
\begin{equation}\label{eq:firstPer}
	\psi_{i,n}^{\rm per}(x):=\sum_{l \in \mathbb{Z}^3} \psi_{i,n}(x+l).
\end{equation}
We observe the following: since $\psi_i$ is a Schwartz function on $\mathbb{R}^3$, we obtain for all $N \in \mathbb{N}$
that
\begin{equation}\label{PsiSchwartz}
	\sup_{x\in\mathbb{T}^3} | \psi_{i, n} (x + l) |^2 (1 + | x + l |)^N \lesssim_N 1 \quad \forall l \in \mathbb{Z}^3,
\end{equation}
hence let $N > 3$, then
\[ \sum_{l \in \mathbb{Z}^3} | \psi_{i, n} (x + l) |^2 \lesssim_N \sum_{l \in\mathbb{Z}^3} (1 + | x + l |)^{- N} < \infty \]
and
\[ \int_{\mathbb{T}^3} (1 + | x + l |)^{- N} {\rm d} x \]
is summable. Therefore \eqref{eq:firstPer} is well-defined and we may exchange integration and summation to obtain for $k \in \mathbb{Z}^3$
\begin{equation}\label{PsiPerPsi}
	\begin{aligned}
  		(\mathcal{F}_{\mathbb{T}^3} \psi^{\rm per}_{i, n}) (k) & = \int_{\mathbb{T}^3} \sum_{l \in \mathbb{Z}^3} \psi_{i, n} (x + l) e^{- 2 \pi {\rm i} k \cdot x} {\rm d} x \\
		& = \sum_{l \in \mathbb{Z}^3} \int_{\mathbb{T}^3} \psi_{i, n} (x + l) e^{- 2 \pi {\rm i} k \cdot x} {\rm d} x \\
		& = \sum_{l \in \mathbb{Z}^3} \int_{\mathbb{T}^3 + l} \psi_{i, n} (z) e^{- 2 \pi {\rm i} k \cdot z} {\rm d} z \\
		& = \int_{\mathbb{R}^3} \psi_{i, n} (z) e^{- 2 \pi {\rm i} k \cdot z}{\rm d} z \\
		& = (\mathcal{F}_{\mathbb{R}^3} \psi_{i, n}) (k)
	\end{aligned}
\end{equation}
where we used $e^{- 2 \pi {\rm i} k \cdot x} = e^{- 2 \pi {\rm i} k\cdot (x+ l)}$ $\forall l \in \mathbb{Z}^3$. Hence
\begin{equation}
	{\rm supp} \mathcal{F}_{\mathbb{T}^3} \psi_{i, n}^{\rm per} =\mathbb{Z}^3 \cap {\rm supp} \mathcal{F}_{\mathbb{R}^3} \psi_{i, n}.
\end{equation}
Furthermore we have 
\begin{equation}\label{PsiINPsi}
	\begin{aligned}
		(\mathcal{F}_{\mathbb{R}^3} \psi_{i, n}) (k) 
		& = (1 + \epsilon_0)^{- \frac{3 n}{2}} (\mathcal{F}_{\mathbb{R}^3}\psi_i) ((1 + \epsilon_0)^{- n} k).
	\end{aligned}
\end{equation}
Thus since ${\rm supp} \mathcal{F}_{\mathbb{R}^3} \psi_i \subset B_i \cup -B_i$, it holds
\begin{equation}\label{PsiINSupp}
	{\rm supp} \mathcal{F}_{\mathbb{R}^3} \psi_{i, n} \subset (1 +\epsilon_0)^n (B_i \cup - B_i)
\end{equation}
and
\begin{equation}\label{PsiPerSuppPsi}
	{\rm supp} \mathcal{F}_{\mathbb{T}^3} \psi_{i, n}^{\rm per} \subset \mathbb{Z}^3 \cap (1 + \epsilon_0)^n (B_i \cup - B_i).
\end{equation}
Finally note that $\psi_{i,n}^{\rm per}$ is divergence-free. Let $\tilde{\psi}^{\rm per}_{i, n}$ denote the $L^2$-normalization of $\psi^{\rm per}_{i, n}$
\[ \tilde{\psi}^{\rm per}_{i, n} (x) := \frac{1}{\| \psi_{i,n}^{\rm per} \|_{L^2 (\mathbb{T}^3)}} \psi^{\rm per}_{i, n} (x) \]
and consider the corresponding Cauchy problem
\begin{equation}\label{eq:periodCauchy}
	\begin{aligned}
		\partial_t u & = \Delta u + C (u, u),\\
		u (0, \cdot) & = u_0 := \tilde{\psi}^{\rm per}_{1, n_0},
	\end{aligned}
\end{equation}
where
\begin{equation*}
	\begin{aligned}
		C (u, v) & := C^{\rm per} (u, v)\\
		&:= \sum_{n, i, \mu} \alpha_{i, \mu} (1 + \epsilon_0)^{\frac{5n}{2}} \langle u, \tilde{\psi}_{i_1, n + \mu_1}^{\rm per} \rangle_{L^2(\mathbb{T}^3)} \langle v, \tilde{\psi}^{\rm per}_{i_2, n + \mu_2}\rangle_{L^2 (\mathbb{T}^3)} \tilde{\psi}^{\rm per}_{i_3, n + \mu_3}.
	\end{aligned}
\end{equation*}
Analogous to Lemma \ref{lem:Tao} define
\begin{align*}
	(\mathcal{F}_{\mathbb{T}^3} u_{i, n} (t)) (k) & := (\mathcal{F}_{\mathbb{T}^3} u (t)) (k) \mathds{1}_{\{ k \in \mathbb{Z}^3 \cap (1 +\epsilon_0)^n (B_i \cup - B_i)\}},\\
	X_{i, n} (t) & := \langle u (t), \tilde{\psi}_{i, n}^{\rm per}\rangle_{L^2},\\
	E_{i, n} (t) & := \frac{1}{2} \|u_{i, n} (t) \|^2_{L^2} .
\end{align*}
First we observe the following: it holds
\begin{equation}\label{uINest2}
	\begin{aligned}
		&\|u_{i,n}(t)\|_{H^{\kappa}} \\
		&= \left\|({\rm Id} - \Delta)^{\frac{\kappa}{2}}u_{i,n}(t)\right\|_{L^2}\\
		&= \left( \sum_{k \in \mathbb{Z}^d} \left|\mathcal{F}_{\mathbb{T}^3}\left(({\rm Id} - \Delta)^{\frac{\kappa}{2}}u_{i,n}(t)\right)\right|^2\right)^{\frac{1}{2}}\\
		&= \left( \sum_{\substack{k \in \mathbb{Z}^3 \cap \\(1+\epsilon_0)^n(B_i \cup - B_i)}}\left(1+4\pi^2|k|^{2n}\right)^{\kappa}\left|\mathcal{F}_{\mathbb{T}^3}u(t)\right|^2\right)^{\frac{1}{2}}\\
		&\lesssim \left(1+4\pi^2(1+\epsilon_0)^{2n}\right)^{-\frac{\beta}{2}}\left( \sum_{\substack{k \in \mathbb{Z}^3 \cap\\ (1+\epsilon_0)^n(B_i \cup - B_i)}}\left(1+4\pi^2|k|^{2n}\right)^{\kappa+\beta}\left|\mathcal{F}_{\mathbb{T}^3}u(t)\right|^2\right)^{\frac{1}{2}}\\
		&\leq \left(1+4\pi^2(1+\epsilon_0)^{2n}\right)^{-\frac{\beta}{2}}\|u(t)\|_{H^{\kappa+\beta}}
	\end{aligned}
\end{equation}
for any $\kappa, \beta \in \mathbb{R}$. Then the blow-up result formulated in Theorem \ref{thm:FGL21divfree} carries over to $\mathbb{T}^3$ as a consequence of the following
\begin{lemma}
	Assume that $u:[0,\infty) \rightarrow H_{\rm df}^{10}(\mathbb{T}^3)$ is a mild solution to \eqref{eq:periodCauchy}, then $(X_{i,n})_{i\in\{1,...,m\},n\in\mathbb{Z}}$ and $(E_{i,n})_{i\in\{1,...,m\},n\in\mathbb{Z}}$ satisfy \eqref{eq:boundXin}-\eqref{eq:initCond} in Lemma \ref{lem:Tao}.
\end{lemma}
\begin{proof}
From \eqref{uINest2} we immediately deduce
\begin{equation*}
	\begin{aligned}
		&\sup_{t\in[0,T]} \sup_{n\in\mathbb{Z}} \sup_{i\in\{1,...,m\}} (1 + (1 + \epsilon_0)^{10 n}) \sqrt{2 E_{i, n} (t)} \\
		& = \sup_{t\in[0,T]} \sup_{n\in\mathbb{Z}} \sup_{i\in\{1,...,m\}} (1 + (1 + \epsilon_0)^{10 n}) \|u_{i, n} (t) \|_{L^2}\\
		& \leq \sup_{t\in[0,T]} \sup_{n\in\mathbb{Z}} (1 + (1 + \epsilon_0)^{10 n}) (1 + 4 \pi^2 (1 +\epsilon_0)^{2 n})^{- 5} \|u (t) \|_{H^{10}}\\
		& \leq\|u\|_{C_t^0 H_x^{10}}
	\end{aligned}
\end{equation*}
as well as
\[ \sup_{t\in[0,T]} \sup_{n\in\mathbb{Z}} \sup_{i\in\{1,...,m\}}  (1 + (1 + \epsilon_0)^{10 n}) | X_{i, n} (t) | \leq \sup_{t,n,i} (1 + (1 + \epsilon_0)^{10 n}) \|u_{i, n} (t) \|_{L^2} \]
which gives \eqref{eq:boundXin} by $u \in C_t^0H_x^{10}$. \\
Analysing the time evolution of the above quantities we obtain for $X_{i, n}$
\[ \partial_t X_{i, n} (t) = \langle \Delta u (t), \tilde{\psi}_{i,n}^{\rm per} \rangle_{L^2} + \langle C (u (t), u (t)),\tilde{\psi}_{i, n}^{\rm per} \rangle_{L^2} . \]
Using \eqref{PsiPerSuppPsi}, the first summand is of the form
\begin{equation*}
	\begin{aligned}
		\langle \Delta u (t), \tilde{\psi}_{i, n}^{\rm per} \rangle_{L^2}& = \sum_{k \in \mathbb{Z}^3}  (\mathcal{F}_{\mathbb{T}^3} (\Delta u(t))) (k) \cdot (\mathcal{F}_{\mathbb{T}^3} \tilde{\psi}_{i,n}^{\rm per}) (k)\\
		& = - 4 \pi^2 \sum_{k \in \mathbb{Z}^3} | k |^2 (\mathcal{F}_{\mathbb{T}^3} u(t)) (k)\cdot (\mathcal{F}_{\mathbb{T}^3} \tilde{\psi}_{i, n}^{\rm per}) (k)\\
		& =- 4 \pi^2 \sum_{k \in \mathbb{Z}^3\cap (1 +\epsilon_0)^n (B_i \cup - B_i)} | k |^2 (\mathcal{F}_{\mathbb{T}^3} u (t)) (k) \cdot (\mathcal{F}_{\mathbb{T}^3} \tilde{\psi}_{i, n}^{\rm per}) (k)\\
		& = - 4 \pi^2 \sum_{k \in \mathbb{Z}^3} | k |^2 (\mathcal{F}_{\mathbb{T}^3} u_{i, n} (t)) (k)\cdot
		 (\mathcal{F}_{\mathbb{T}^3}\tilde{\psi}_{i, n}^{\rm per}) (k)\\
		& = \langle \Delta u_{i, n} (t), \tilde{\psi}_{i, n}^{\rm per}\rangle_{L^2}.
	\end{aligned}
\end{equation*}
Similarly and since
\begin{equation*}
	\begin{aligned}
		&(\mathcal{F}_{\mathbb{T}^3} C (u (t), u (t))) (k) \\
		& = \sum_{n, i, \mu}\alpha_{i, \mu} (1 + \epsilon_0)^{\frac{5 n}{2}} \langle u (t), \tilde{\psi}_{i_1, n + \mu_1}^{\rm per}\rangle_{L^2} \langle u (t), \tilde{\psi}_{i_2, n + \mu_2}^{\rm per} \rangle_{L^2} (\mathcal{F}_{\mathbb{T}^3} \tilde{\psi}_{i_3, n + \mu_3}^{\rm per}) (k)
	\end{aligned}
\end{equation*}
we obtain
\begin{equation*}
	\begin{aligned}
		&\langle C (u (t), u (t)), \tilde{\psi}_{i, n}^{\rm per} \rangle_{L^2} \\
		& = \sum_{k \in \mathbb{Z}^3} \sum_{\tilde{n}, j, \mu}\alpha_{j, \mu} (1 + \epsilon_0)^{\frac{5 \tilde{n}}{2}}\langle u (t), \tilde{\psi}_{j_1, \tilde{n} +\mu_1}^{\rm per} \rangle_{L^2} \langle u (t), \tilde{\psi}_{j_2, \tilde{n} + \mu_2}^{\rm per} \rangle_{L^2}\\
		& \hspace{2.5cm}  (\mathcal{F}_{\mathbb{T}^3} \tilde{\psi}_{j_3,\tilde{n} + \mu_3}^{\rm per}) (k) \cdot (\mathcal{F}_{\mathbb{T}^3} \tilde{\psi}_{i,n}^{\rm per}) (k)\\
		& =\sum_{k \in \mathbb{Z}^3} \sum_{i_1, i_2, \mu} \alpha_{i_1, i_2, i,\mu} (1 + \epsilon_0)^{\frac{5 (n - \mu_3)}{2}}\\
		& \hspace{2cm} \langle u (t), \tilde{\psi}_{i_1, n - \mu_3 +\mu_1}^{\rm per} \rangle_{L^2} \langle u (t), \tilde{\psi}_{i_2, n - \mu_3 + \mu_2}^{\rm per} \rangle_{L^2} \left|(\mathcal{F}_{\mathbb{T}^3} \tilde{\psi}_{i,n}^{\rm per}) (k)\right|^2\\
		& =\sum_{i_1, i_2, \mu} \alpha_{i_1, i_2, i, \mu} (1 +\epsilon_0)^{\frac{5 (n - \mu_3)}{2}} X_{i_1, n - \mu_3 + \mu_1} (t) X_{i_2,n - \mu_3 + \mu_2} (t)
	\end{aligned}
\end{equation*}
which in total yields
\begin{equation*}
	\begin{aligned}
		\partial_t X_{i, n} (t) & = \langle \Delta u_{i, n} (t), \tilde{\psi}_{i,n}^{\rm per} \rangle_{L^2}\\
		& \hspace{0.5cm} + \sum_{i_1, i_2, \mu} \alpha_{i_1, i_2, i, \mu} (1 +\epsilon_0)^{\frac{5 (n - \mu_3)}{2}} X_{i_1, n - \mu_3 + \mu_1} (t) X_{i_2,n - \mu_3 + \mu_2} (t).
	\end{aligned}
\end{equation*}
Furthermore by \eqref{uINest2} it holds
\[ \langle \Delta u_{i, n}(t), \tilde{\psi}_{i, n}^{\rm per} \rangle_{L^2} \in O\left( (1 + \epsilon_0)^{2 n}\sqrt{E_{i, n}(t)} \right) \]
which yields \eqref{eq:evolXin}. Additionally we have
\begin{equation*}
	\begin{aligned}
		X_{i, n} (0) & = \langle u (0), \tilde{\psi}_{i, n}^{\rm per} \rangle_{L^2} = \langle u_0, \tilde{\psi}_{i, n}^{\rm per} \rangle_{L^2} = \langle \tilde{\psi}_{1, n_0}^{\rm per}, \tilde{\psi}_{i, n}^{\rm per}\rangle_{L^2}\\
		& =\mathds{1}_{\{(i, n) = (1, n_0)\}}.
	\end{aligned}
\end{equation*}
For the local energy we obtain by a similar analysis as for $X_{i, n}$:
\begin{equation*}
	\begin{aligned}
		&\partial_t E_{i, n} (t) \\
		& =  \sum_{k \in \mathbb{Z}^3}(\mathcal{F}_{\mathbb{T}^3} (\Delta u_{i, n} (t))) (k)\cdot(\mathcal{F}_{\mathbb{T}^3} u_{i, n} (t)) (k)\\
		& \hspace{1cm} + \sum_{i_1, i_2, \mu} \alpha_{i_1, i_2, i, \mu} (1 +\epsilon_0)^{\frac{5 (n - \mu_3)}{2}} X_{i_1, n - \mu_3 + \mu_1} (t) X_{i_2,n - \mu_3 + \mu_2} (t)\\
		& \hspace{2.5cm} (\mathcal{F}_{\mathbb{T}^3} \tilde{\psi}_{i, n}^{\rm per})(k)\cdot (\mathcal{F}_{\mathbb{T}^3} u_{i, n} (t)) (k)\\
		& =\langle \Delta u_{i, n} (t), u_{i, n} (t) \rangle_{L^2}\\
		& \hspace{0.5cm}+ \sum_{i_1, i_2, \mu} \alpha_{i_1, i_2, i, \mu} (1 +\epsilon_0)^{\frac{5 (n - \mu_3)}{2}} X_{i_1, n - \mu_3 + \mu_1} (t) X_{i_2,n - \mu_3 + \mu_2} (t) X_{i, n} (t).
	\end{aligned}
\end{equation*}
The rest of the proof is analogous to the proof of Lemma \ref{lem:Tao} in \cite{tao2016}.
\end{proof}
\subsection{Local well-posedness}\label{subsec:localW}
Observe that the analysis in \cite{flandoli2021b} requires the existence of unique local solutions both to \eqref{eq:generalDet} as well as \eqref{eq:regFGL21} and \eqref{eq:regFGL21divfree} (cf. Remark 1.3 (i) in \cite{flandoli2021b}). Consider systems of the form
\begin{equation}\label{eq:periodCauchy}
	\partial_t u  = \Delta u + F(u)
\end{equation}
where
\begin{equation}\label{eq:nonlin}
	\begin{aligned}
		F(u)&:= C (u, u) \\
		&= \sum_{n, i, \mu} \alpha_{i, \mu} (1 + \epsilon_0)^{\frac{5n}{2}} \langle u, \tilde{\psi}_{i_1, n + \mu_1}^{\rm per} \rangle_{L^2(\mathbb{T}^3)} \langle u, \tilde{\psi}^{\rm per}_{i_2, n + \mu_2}\rangle_{L^2 (\mathbb{T}^3)} \tilde{\psi}^{\rm per}_{i_3, n + \mu_3}.
	\end{aligned}
\end{equation}
\noindent
Recall that in Tao's model, the coefficients $\alpha_{i,\mu}$ are chosen in such a way that the cancellation property
\begin{equation}
	\langle F(u), u \rangle_{L^2} = 0
\end{equation}
holds for all $u \in H_{\rm df}^{10}$. Thus if $u \in H_{\rm df}^{10}$, then we easily deduce
\begin{equation}
	\partial_t \|u(t)\|_{L^2}^2 = -\|\nabla u(t)\|_{L^2}^2
\end{equation}
and hence the energy equality
\begin{equation}
	\sup_{t \in [0,T]} \|u(t)\|_{L^2}^2 + \int_0^T\|\nabla u(t)\|_{L^2}^2{\rm d}t = \|u_0\|_{L^2}^2
\end{equation}
for $T \in [0,\infty]$. Thus let us define a weak solution to \eqref{eq:periodCauchy} in the following way:
\begin{definition}
	A vector field $u \in L^{\infty}(0,T;L^2_{\rm df}(\mathbb{T}^3)) \cap L^2(0,T; H^1_{\rm df}(\mathbb{T}^3))$ is called a weak solution to \eqref{eq:periodCauchy} if
	\begin{equation}
		\begin{aligned}
			&-\int_0^T \int_{\mathbb{T}^3} \langle u, \partial_t \phi\rangle {\rm d}x {\rm d}t - \int_0^T \int_{\mathbb{T}^3} \langle F(u), \phi\rangle{\rm d}x {\rm d}t +\int_0^T\int_{\mathbb{T}^3} \langle \nabla u , \nabla \phi \rangle {\rm d}x {\rm d}t\\
			&= \int_{\mathbb{T}^3} \langle u_0, \phi(0) {\rm d}x
		\end{aligned}
	\end{equation}
	for any divergence-free test function $\phi \in C^{\infty}_c([0,T)\times\mathbb{T}^3)$. 
\end{definition}
\noindent
In this section, we shall discuss whether there exist unique weak solutions to \eqref{eq:periodCauchy}. In general, local existence and uniqueness are guaranteed by the hypotheses (H1)-(H3) roughly as follows: first considering a Galerkin approximation on a finite-dimensional subspace, (H1) and (H3) provide that locally, corresponding solutions exist and are unique. Moreover by (H2), they satisfy an energy inequality, and with the help of (H1) again we may pass to the limit to recover local unique solutions for the original system. 
\subsubsection{Violation of hypotheses}
\noindent
It turns out, however, that for $u \in L^{\infty}(0,T;L^2_{\rm df}(\mathbb{T}^3)) \cap L^2(0,T; H^1_{\rm df}(\mathbb{T}^3))$ neither of the hypotheses is satisfied:
\begin{lemma}\label{lem:violation}
	The operator $F$ as defined in \eqref{eq:nonlin} does not satisfy (H1)-(H3).
\end{lemma}
\begin{proof}
	In attempting to prove the hypotheses, the procedure is as follows: in order to estimate terms of the form $|\langle F(u), \phi\rangle|$, we need to first justify the interchange of integration and summation over $n \in \mathbb{Z}$, i.e. that
	\[\sum_{n, i, \mu}\left| \alpha_{i, \mu} (1 + \epsilon_0)^{\frac{5n}{2}} \langle u, \tilde{\psi}_{i_1, n + \mu_1}^{\rm per} \rangle_{L^2} \langle u, \tilde{\psi}^{\rm per}_{i_2, n + \mu_2}\rangle_{L^2} \langle \tilde{\psi}^{\rm per}_{i_3, n + \mu_3}, \phi \rangle\right|\]
	is well-defined. We consider the sums over $n<0$ and $n \geq 0$ separately: in the former case since the factor $(1 + \epsilon_0)^{\frac{5n}{2}}$ is already summable for $n<0$, we may crudely estimate terms of the form $\langle v, \tilde{\psi}^{\rm per}_{i, n}\rangle_{L^2}$ by $\|v\|_{H^{\kappa}}$ for any $\kappa \geq 0$ using that the functions $\tilde{\psi}^{\rm per}_{i, n}$ are $L^2$-normalized. In the case of $n \in \mathbb{N}_0$, instead estimate via the observation \eqref{uINest2} to compensate the in this case diverging factor $(1 + \epsilon_0)^{\frac{5n}{2}}$.\\
	Violation of (H1):\\
	Let $u_1, u_2 \in H^{1-\eta}(\mathbb{T}^3)$ and $\phi \in H^1(\mathbb{T}^3)$, then in view of $|\langle F(u_1) - F(u_2), \phi \rangle|$ we estimate the summands in
	\begin{equation*}
		\begin{aligned}
			 &\sum_{n,i,\mu} \left|\alpha_{i,\mu} (1+\epsilon_0)^{\frac{5n}{2}}\langle \tilde{\psi}^{\rm per}_{i_3, n+\mu_3}, \phi \rangle\right.\\
			&\qquad\left.\left( \langle u_1 - u_2, \tilde{\psi}^{\rm per}_{i_1, n+\mu_1}\rangle_{L^2} \langle u_1, \tilde{\psi}^{\rm per}_{i_2, n+ \mu_2}\rangle_{L^2}  \right.\right.\\
			&\qquad\qquad \left.\left.+ \langle u_2, \tilde{\psi}^{\rm per}_{i_1, n+\mu_1}\rangle_{L^2} \langle u_1 - u_2, \tilde{\psi}^{\rm per}_{i_2, n+\mu_2}\rangle_{L^2} \right)\right|.
		\end{aligned}
	\end{equation*}
	For $n \in \mathbb{N}_0$, we estimate
	\begin{equation}
		\begin{aligned}
			&\left|(1 + \epsilon_0)^{\frac{5n}{2}} \langle u_1 - u_2, \tilde{\psi}_{i_1, n + \mu_1}^{\rm per} \rangle_{L^2} \langle u_1, \tilde{\psi}^{\rm per}_{i_2, n + \mu_2}\rangle_{L^2} \langle \tilde{\psi}^{\rm per}_{i_3, n + \mu_3}, \phi \rangle\right|\\
			&\lesssim \left(1+ 4\pi^2(1+\epsilon_0)^{2n}\right)^{\frac{1}{4}-\eta}\|u_1 -u_2\|_{H^{1-\eta}}\|u_1\|_{H^{1-\eta}}\|\phi\|_{H^1}
		\end{aligned}
	\end{equation}
	which is summable for $\eta < \frac{1}{4}$. However in view of the second claim on $\|F(u)\|_{H^{-1}(\mathbb{T}^3)}$, we obtain
	\begin{equation}
		\begin{aligned}
			&\left|(1 + \epsilon_0)^{\frac{5n}{2}} \langle u, \tilde{\psi}_{i_1, n + \mu_1}^{\rm per} \rangle_{L^2} \langle u, \tilde{\psi}^{\rm per}_{i_2, n + \mu_2}\rangle_{L^2} \langle \tilde{\psi}^{\rm per}_{i_3, n + \mu_3}, \phi \rangle\right|\\
			&\lesssim \left(1+ 4\pi^2(1+\epsilon_0)^{2n}\right)^{\frac{1}{4}}\|u\|_{L^2}\|u\|_{H^1}\|\phi\|_{H^1}
		\end{aligned}
	\end{equation}
	which is not summable over $n \in \mathbb{N}_0$.\\
	Violation of (H2):\\
	Observe that even the more general form (H2') in Remark \ref{rem:hypFGL21}, (3), is not satisfied: let $\alpha, \beta, \gamma \in [0,1]$, then interpolation gives an estimate of the form
	\[\left(1+ 4\pi^2(1+\epsilon_0)^{2n}\right)^{\frac{5}{4}-\frac{1}{2}(\alpha + \beta +\gamma)}\|u\|^{\frac{3-(\alpha + \beta + \gamma)}{1+\delta}}_{H^{-\delta}}\|u\|^{\frac{3\delta + \alpha + \beta + \gamma}{1+\delta}}_{H^1}\|\phi\|_{H^1}\]
	for which (H2') requires
	\[\frac{3\delta +\alpha + \beta + \gamma}{1+\delta} \in (0,2)\]
	whereas for summability we need $\alpha + \beta +\gamma > \frac{5}{2}$ yielding
	\[\frac{3\delta + \alpha + \beta + \gamma}{1+\delta} > 3 - \frac{1}{2(1+\delta)} > 2.\]
	Violation of (H3):\\
	We show that also here the more general form (H3') (see Remark \ref{rem:hypFGL21}, (4)) is violated: let $\gamma \in [0,1]$, then we first estimate
	\begin{equation}\label{eq:estH3}
	|\langle u_1 - u_2, F(u_1) - F(u_2)\rangle| \leq \|u_1 -u_2 \|_{H^{\gamma}}\|F(u_1) - F(u_2) \|_{H^{-\gamma}}.
	\end{equation}
	Let $\alpha, \beta \in [0,1]$, then similar to our analysis for (H1) we estimate via interpolation
	\begin{equation}
		\begin{aligned}
			&\left|(1 + \epsilon_0)^{\frac{5n}{2}} \langle u_1 - u_2, \tilde{\psi}_{i_1, n + \mu_1}^{\rm per} \rangle_{L^2} \langle u_1, \tilde{\psi}^{\rm per}_{i_2, n + \mu_2}\rangle_{L^2} \langle \tilde{\psi}^{\rm per}_{i_3, n + \mu_3}, \phi \rangle\right|\\
			&\lesssim \left(1+ 4\pi^2(1+\epsilon_0)^{2n}\right)^{\frac{5}{4}-\frac{1}{2}(\alpha + \beta +\gamma)}\\
			&\quad\|u_1 -u_2\|^{1-\alpha}_{L^2}\|u_1 -u_2\|^{\alpha}_{H^1}\|u_1\|^{1-\beta}_{L^2}\|u_1\|^{\beta}_{H^1}\|\phi\|_{H^1}
		\end{aligned}
	\end{equation}
	hence summability requires again $\alpha + \beta +\gamma > \frac{5}{2}$. Together with \eqref{eq:estH3} we obtain a total estimate of the form
	\begin{equation}
		\begin{aligned}
			&|\langle u_1 - u_2, F(u_1) - F(u_2)\rangle|\\
			&\lesssim \|u_1 -u_2\|^{2-(\alpha+\gamma)}_{L^2}\|u_1 -u_2\|^{\alpha+\gamma}_{H^1}\left(\|u_1\|^{1-\beta}_{L^2}\|u_1\|^{\beta}_{H^1} + \|u_2\|^{1-\beta}_{L^2}\|u_2\|^{\beta}_{H^1}\right).
		\end{aligned}
	\end{equation}
	Hypotheses (H3') hence requires in particular that
	\[ (\alpha +\gamma) + \beta_2 \leq 2\]
	which however violates the above summability condition
	\[ 0> \frac{5}{4}-\frac{1}{2}(\alpha + \beta +\gamma) \geq \frac{1}{4}.\]
\end{proof}
\noindent
The take-away message from this proof is that though at first sight the cascade operators are of seemingly simple structure, it is the factor $(1 + \epsilon_0)^{\frac{5n}{2}}$ that dictates whether one may deduce the desired estimates. Note that this factor encodes the relation of the cascade operator to the Euler bilinear operator $B$ (at least in a dyadic framework as in \cite{katz2005}) and mimics its scaling behaviour. Furthermore our analysis works irrespective of the precise form of the coefficients $\alpha_{i,\mu}$ whereas in \cite{tao2016} these parameters are carefully chosen so as to facilitate the blow-up.
\begin{remark}
A similar behaviour can be observed in the case of standard NSE: consider
\begin{equation}
	\partial_t u = \Delta u + F(u)
\end{equation}
with $F(u) = B(u,u)=-\Pi((u\cdot \nabla)u)$, then we may investigate the hypotheses with the help of \cite[Lemma 2.1]{temam1995} stating that
\[|\langle B(u,v), w\rangle| \lesssim \|u\|_{H^{m_1}}\|v\|_{H^{m_2+1}}\|w\|_{H^{m_3}}\]
where
\begin{equation}\label{eq:requTemam}
\frac{3}{2} \leq m_1 + m_2 + m_3,\quad 0\leq m_i \neq \frac{3}{2}, i=1,2,3.
\end{equation}
Violation of (H1) follows immediately. For (H2') observe that by interpolation and using $m_2 =0$ we may estimate
\begin{equation}
	|\langle B(u,u),u\rangle| \leq \|u\|_{H^{-\delta}}^{\frac{2-(m_1+m_3)}{1+\delta}}\|u\|_{H^1}^{1+\frac{m_1+m_3+2\delta}{1+\delta}}
\end{equation}
where (H2') requires
\[1+\frac{m_1+m_3+2\delta}{1+\delta} < 2 \quad \Rightarrow \quad m_1 + m_3 < 1 - \delta\]
which is in conflict with the requirement \eqref{eq:requTemam}. In the case of (H3'), for two divergence-free vector fields $u_1, u_2$ we use the identity
\begin{equation}\label{eq:stanNSEident}
	|\langle u_1-u_2, F(u_1)- F(u_2) \rangle| = |\langle u_1-u_2, (u_1\cdot \nabla)(u_1-u_2) + ((u_1-u_2)\cdot \nabla)u_2 \rangle|.
\end{equation}
Then by \cite[Lemma 2.1]{temam1995} we estimate
\begin{equation}
	|\langle (u_1\cdot \nabla)(u_1-u_2) , u_1-u_2 \rangle |\leq \|u_1\|_{H^{m_1}}\|u_1-u_2\|_{H^{m_2+1}}\|u_1-u_2\|_{H^{m_3}}
\end{equation}
hence we require $m_2 =0$. By interpolation it holds
\begin{equation}
	|\langle (u_1\cdot \nabla)(u_1-u_2) , u_1-u_2 \rangle |\leq \|u_1\|_{L^2}^{1-m_1}\|u_1\|_{H^1}^{m_1}\|u_1-u_2\|_{L^2}^{1-m_3}\|u_1-u_2\|_{H^1}^{1+m_3}
\end{equation}
where (H3') requires
\[(1+m_3) + m_1 \leq 2 \quad \Rightarrow \quad m_1 + m_3 \leq 1\]
violating \eqref{eq:requTemam}. One proceeds similarly for the second summand.
\end{remark}

\subsubsection{Order of well-posed derivatives}
\noindent
Denote $v := ({\rm Id} - \Delta)^{\rho} u$, then $v$ satisfies
\begin{equation}\label{eq:sysRho}
	\partial_t v = \Delta v + F_{\rho}(v)
\end{equation}
where
\begin{equation}
	F_{\rho}(v) := ({\rm Id}-\Delta)^{\rho} F(u).
\end{equation}
This section shall discuss the minimal threshold value of $\rho >0$ for which $F_{\rho}$ is well-defined in the sense that
\begin{equation}
	F_{\rho}(v) = \sum_{n , i, \mu}\alpha_{i,\mu} (1+\epsilon_0)^{\frac{5n}{2}}\langle u, \tilde{\psi}_{i_1, n + \mu_1}^{\rm per} \rangle_{L^2} \langle u, \tilde{\psi}^{\rm per}_{i_2, n + \mu_2}\rangle_{L^2} ({\rm Id}-\Delta)^{\rho} \tilde{\psi}^{\rm per}_{i_3, n + \mu_3}
\end{equation}
and attains the hypotheses from Section \ref{subsubsec:criteriaBlowUp}. First recall that it holds
\begin{equation}\label{eq:estV}
	\|v\|_{H^{\gamma}} =\left\|({\rm Id} - \Delta)^{\frac{\gamma}{2}}v\right\|_{L^2}=\left\|({\rm Id}-\Delta)^{\rho + \frac{\gamma}{2}}u\right\|_{L^2} = \|u\|_{H^{2\rho+\gamma}}
\end{equation}
and denote $v_i= ({\rm Id}-\Delta)^{\rho}u_i$, $i=1,2$. 
\begin{lemma}
	$F_{\rho}$ satisfies (H1), (H2') and (H3') if $\rho > \frac{1}{8}$.
\end{lemma}
\begin{proof}
	In case of (H1), claim 2, we estimate for $n \in \mathbb{N}_0$
	\begin{equation}
		\begin{aligned}
			&\left|(1+\epsilon_0)^{\frac{5n}{2}}\langle  u,\tilde{\psi}^{\rm per}_{i_1, n+\mu_1}\rangle_{L^2} \langle  u,\tilde{\psi}^{\rm per}_{i_2, n+\mu_2}\rangle_{L^2} \langle  ({\rm Id}-\Delta)^{\rho}\tilde{\psi}^{\rm per}_{i_3, n+\mu_3}, \phi \rangle_{L^2}\right|\\
			& \lesssim \left(1+4\pi^2(1+\epsilon_0)^{2n}\right)^{\frac{5}{4}-(2\rho +1)}\|u\|_{H^{2\rho}}\|u\|_{H^{2\rho +1}}\|\phi\|_{H^1}
		\end{aligned}
	\end{equation}
	which is summable if
	\[\rho > \frac{1}{8}.\]
	For (H2') we use interpolation as in the proof of Lemma \ref{lem:violation} to obtain the estimate
	\begin{equation}
		\begin{aligned}
			&\left|(1 + \epsilon_0)^{\frac{5n}{2}} \langle u, \tilde{\psi}_{i_1, n + \mu_1}^{\rm per} \rangle_{L^2} \langle u, \tilde{\psi}^{\rm per}_{i_2, n + \mu_2}\rangle_{L^2 } \langle ({\rm Id}-\Delta)^{\rho} \tilde{\psi}^{\rm per}_{i_3, n + \mu_3}, u \rangle\right|\\
			&\lesssim \left(1+ 4\pi^2(1+\epsilon_0)^{2n}\right)^{\frac{5}{4}-\frac{1}{2}(\alpha + \beta +\gamma)}\|u\|_{H^{\alpha} }\|u\|_{H^{\beta}}\|({\rm Id}-\Delta)^{\rho}u\|_{H^{\gamma} }\\
			&\leq \left(1+ 4\pi^2(1+\epsilon_0)^{2n}\right)^{\frac{5}{4}-\frac{1}{2}(\alpha + \beta +\gamma)}\|v\|^{\frac{3-(\alpha + \beta + \gamma-4\rho)}{1+\delta}}_{H^{-\delta}}\|u\|^{\frac{3\delta + \alpha + \beta + \gamma-4\rho}{1+\delta}}_{H^1}
		\end{aligned}
	\end{equation}
	which is summable if $\alpha + \beta +\gamma > \frac{5}{2}$ and satisfies the requirements of (H2') if
	\[ \rho > \frac{1}{4}\left( \delta + \alpha + \beta + \gamma -2\right) > \frac{1}{4}\left(\delta + \frac{1}{2}\right) > \frac{1}{8}.\]
	Finally for (H3'), let $\gamma \in [0,1]$, then we first estimate
	\begin{equation}\label{eq:estH3v}
		|\langle v_1 - v_2, F_{\rho}(v_1) - F_{\rho}(v_2)\rangle| \leq \|v_1 -v_2 \|_{H^{\gamma}}\|F_{\rho}(v_1) - F_{\rho}(v_2) \|_{H^{-\gamma}}.
	\end{equation}
	We continue as in the proof of Lemma \ref{lem:violation}: in the case of $n \in \mathbb{N}_0$, let $\alpha$ and $\beta$ be such that $\alpha - 2 \rho, \beta - 2\rho \in [0,1]$. Then via interpolation we obtain for $\phi \in H^{\gamma}(\mathbb{T}^3)$
	\begin{equation}
		\begin{aligned}
			&\left|(1+\epsilon_0)^{\frac{5n}{2}}\langle u_1 - u_2, \tilde{\psi}^{\rm per}_{i_1, n+\mu_1}\rangle_{L^2} \langle u_1, \tilde{\psi}^{\rm per}_{i_2, n+ \mu_2}\rangle_{L^2}\langle  ({\rm Id}-\Delta)^{\rho}\tilde{\psi}^{\rm per}_{i_3, n+\mu_3}, \phi \rangle\right|\\
			&\lesssim \left(1+4\pi^2(1+\epsilon_0)^{2n}\right)^{\frac{5}{4}-\frac{1}{2}(\alpha + \beta+\gamma-2\rho)}\|u_1 - u_2 \|_{H^{\alpha}} \|u_1\|_{H^{\beta}} \|\phi\|_{H^{\gamma}}\\
			&\leq \left(1+4\pi^2(1+\epsilon_0)^{2n}\right)^{\frac{5}{4}-\frac{1}{2}(\alpha + \beta+\gamma-2\rho)}\\
			&\quad\|v_1 - v_2 \|^{1-(\alpha-2\rho)}_{L^2} \|v_1-v_2\|^{\alpha -2\rho}_{H^1}\|v_1\|^{1-(\beta -2\rho)}_{L^2}\|v_1\|^{\beta -2\rho}_{H^1} \|\phi\|_{H^{\gamma}}
		\end{aligned}
	\end{equation}
	which is summable if $\alpha + \beta +\gamma > \frac{5}{2}$. Thus we obtain
	\begin{equation}
		\begin{aligned}
			&|\langle v_1 - v_2, F_{\rho}(v_1) - F_{\rho}(v_2)\rangle|\\
			&\lesssim \|v_1 - v_2 \|^{2-(\alpha+\gamma-2\rho)}_{L^2} \|v_1-v_2\|^{\alpha+\gamma -2\rho}_{H^1}\\
			&\quad\left(\|v_1\|^{1-(\beta -2\rho)}_{L^2}\|v_1\|^{\beta -2\rho}_{H^1} + \|v_2\|^{1-(\beta -2\rho)}_{L^2}\|v_2\|^{\beta -2\rho}_{H^1}\right)
		\end{aligned}
	\end{equation}
	which satisfies the requirements in (H3') if
		\[ 2(1-2\rho) \geq \alpha + \beta +\gamma > \frac{5}{2} \quad \Rightarrow \quad \rho > \frac{1}{8}.\]
\end{proof}
\begin{remark}
	In case of the standard NSE, using \cite[Lemma 2.1]{temam1995} we estimate for (H1), claim 2,
	\begin{equation}
		\begin{aligned}
			|\langle ({\rm Id}-\Delta)^{\rho} B(u,u), \phi\rangle | &\leq \|u\|_{H^{m_1}}\|u\|_{H^{m_2+1}}\|({\rm Id}-\Delta)^{\rho}\phi\|_{H^{m_3}} \\
			&= \|v\|_{H^{m_1-2\rho}}\|v\|_{H^{m_2+1-2\rho}}\|\phi\|_{H^{m_3+2\rho}}.
		\end{aligned}
	\end{equation}
	Hence we require $m_3 = 1-2\rho$. Since claim 2 in (H1) requires for an estimate involving the $L^2(\mathbb{T}^3)$- and the $H^1(\mathbb{T}^3)$-norm, assume $m_1 - 2\rho \in [0,1], m_2+1-2\rho \in [0,1]$. Via interpolation we obtain
	\[ \|v\|_{H^{m_1-2\rho}}\|v\|_{H^{m_2+1-2\rho}} \lesssim \|v\|_{L^2}^{1+4\rho - (m_1+m_2)}\|v\|_{H^1}^{1-4\rho +m_1 + m_2}\]
	and claim 2 in (H1) requires
	\[ 1-4\rho +m_1 + m_2 = 1 \quad \Rightarrow \quad 4\rho = m_1 + m_2.\]
	Further with the requirement in \cite[Lemma 2.1]{temam1995} we obtain
	\[ \frac{3}{2} \leq m_1 + m_2 + m_3 = 1 + 2\rho \quad \Rightarrow \quad \frac{1}{4} \leq \rho.\]
	This threshold equally holds in the case of (H2') and (H3') via a similar analysis.
\end{remark}
In terms of regularizability as specified by Theorem \ref{thm:FGL21divfree}, we finally need to check whether also (H4) is satisfied:
\begin{lemma}
	If $\rho > \frac{1}{4}$, then we obtain (H4) with $\mathcal{K}$ as in Remark \ref{rem:hypFGL21}, \eqref{eq:H4K}.
\end{lemma}
\begin{proof}
In view of Remark \ref{rem:hypFGL21}, \eqref{eq:H4K}, we show that the system \eqref{eq:sysRho} preserves the set of mean-zero functions in $L^2(\mathbb{R}^3)$. More precisely we show that $F_{\rho}$ has zero mean: first observe that
\begin{equation}
	\begin{aligned}
		\int_{\mathbb{T}^3} \left( ({\rm Id}-\Delta)^{\rho} \tilde{\psi}^{\rm per}_{i_3, n + \mu_3}\right)(x){\rm d} x &= \frac{1}{2^3} \mathcal{F}_{\mathbb{T}^3}\left(({\rm Id}-\Delta)^{\rho} \tilde{\psi}^{\rm per}_{i_3, n + \mu_3}\right)(0)\\
		&= \frac{1}{2^3} \left((1+4\pi^2|\cdot|^2)^{\rho} \mathcal{F}_{\mathbb{T}^3}\tilde{\psi}^{\rm per}_{i_3, n + \mu_3}\right)(0)\\
		&= \frac{1}{2^3} \left(\mathcal{F}_{\mathbb{T}^3}\tilde{\psi}^{\rm per}_{i_3, n + \mu_3}\right)(0) =0
	\end{aligned}
\end{equation}
since $\mathcal{F}_{\mathbb{T}^3}\tilde{\psi}^{\rm per}_{i, n}$ are supported away from 0. In order to justify that we may interchange integration and derivatives with summation, further observe that it holds with the help of \eqref{uINest2}
\begin{equation}
	\begin{aligned}
		\int_{\mathbb{T}^3} \left|\left( ({\rm Id}-\Delta)^{\rho} \tilde{\psi}^{\rm per}_{i_3, n + \mu_3}\right)(x)\right|{\rm d} x &\leq \left\| \tilde{\psi}^{\rm per}_{i_3, n + \mu_3}\right\|_{H^{2\rho}}\\
		&\lesssim \left(1 + 4\pi^2 (1+\epsilon_0)^{2n}\right)^{\rho}.
	\end{aligned}
\end{equation}
Hence using $u(t) \in H^1(\mathbb{T}^3)$ together with \eqref{uINest2} we estimate
\begin{equation}
	\begin{aligned}
		&(1+\epsilon_0)^{\frac{5n}{2}}|\langle u, \tilde{\psi}_{i_1, n + \mu_1}^{\rm per} \rangle_{L^2}| |\langle u, \tilde{\psi}^{\rm per}_{i_2, n + \mu_2}\rangle_{L^2}|\int_{\mathbb{T}^3}\left| \left( ({\rm Id}-\Delta)^{\rho} \tilde{\psi}^{\rm per}_{i_3, n + \mu_3}\right)(x)\right|{\rm d} x\\
		&\lesssim \left(1 + 4\pi^2 (1+\epsilon_0)^{2n}\right)^{\frac{1}{4} - \rho}.
	\end{aligned}
\end{equation}
Hence for $\rho > \frac{1}{4}$ we justified
\begin{equation}
	\begin{aligned}
		\int_{\mathbb{T}^3} F_{\rho}(v)(x){\rm d}x &= \int_{\mathbb{T}^3} \sum_{n , i, \mu}\alpha_{i,\mu} (1+\epsilon_0)^{\frac{5n}{2}}\langle u, \tilde{\psi}_{i_1, n + \mu_1}^{\rm per} \rangle_{L^2} \langle u, \tilde{\psi}^{\rm per}_{i_2, n + \mu_2}\rangle_{L^2}\\
		&\hspace{2cm}\times\left(({\rm Id}-\Delta)^{\rho} \tilde{\psi}^{\rm per}_{i_3, n + \mu_3}\right)(x) {\rm d}x\\
		&= \sum_{n , i, \mu}\alpha_{i,\mu} (1+\epsilon_0)^{\frac{5n}{2}}\langle u, \tilde{\psi}_{i_1, n + \mu_1}^{\rm per} \rangle_{L^2} \langle u, \tilde{\psi}^{\rm per}_{i_2, n + \mu_2}\rangle_{L^2}\\
		&\hspace{1cm}\times\int_{\mathbb{T}^3} \left( ({\rm Id}-\Delta)^{\rho} \tilde{\psi}^{\rm per}_{i_3, n + \mu_3}\right)(x){\rm d} x =0.
	\end{aligned}
\end{equation}
\end{proof}
\subsection{Regularization of the periodic averaged NSE}
\noindent
Recall that in the case of the periodic averaged NSE we consider the system \eqref{eq:periodCauchy} given by
\begin{equation}
	\begin{aligned}
		\partial_t u & = \Delta u + C (u, u),\\
		u (0, \cdot) & = \tilde{\psi}^{\rm per}_{1, n_0},
	\end{aligned}
\end{equation}
and the corresponding blow-up result holds in the setting of $H_{\rm df}^{10}(\mathbb{T}^3)$. Let 
\[ v = ({\rm Id} - \Delta)^5u,\]
then on the one hand clearly it holds
\[ \|v(t)\|_{L^2} = \|({\rm Id} - \Delta)^5u(t)\|_{L^2} = \|u(t)\|_{H^{10}},\]
and on the other hand (H1)-(H4) hold as seen in the previous section. Theorem \ref{thm:FGL21divfree} therefore implies that for arbitrary large time horizon $T$ the solution to
\begin{equation}
	\begin{aligned}
		{\rm d} v &= (\Delta v + F_5(v)){\rm d} t + \frac{\sqrt{C_3 \nu}}{\|\theta\|_{\ell^2}}\sum_{k \in \mathbb{Z}_0^3}\sum_{i=1}^{2}\theta_k \Pi((\sigma_{k,i}\cdot\nabla)v) \circ {\rm d}W^{k,i},\\
		v(0,\cdot) &= ({\rm Id} - \Delta)^5\tilde{\psi}^{\rm per}_{1, n_0},
	\end{aligned}
\end{equation}
does not blow up in $C([0,T];L^2(\mathbb{T}^3))$ with high probability.

\appendix
\section{Proof of Theorem \ref{thm:FGL21divfree}}\label{app:proofThmDivfree}
\noindent
The proof closely follows \cite{flandoli2021b} making use of results from \cite{flandoli2021a}, thereby using the same or similar notation: let $T, R >0$ be fixed parameters. As in \cite{flandoli2021a} and \cite{flandoli2021b}, we first consider the cut-off equation
\begin{equation}\label{eq:StratDivfree}
	\begin{aligned}
		{\rm d}u &= \left( \Lambda^{2\alpha}u + g_R(u)F(u)\right){\rm d}t + \frac{\sqrt{C_d \nu}}{\|\theta\|_{\ell^2}}\sum_{k \in \mathbb{Z}_0^d}\sum_{i=1}^{d-1}\theta_k \Pi((\sigma_{k,i}\cdot\nabla)u)\circ {\rm d}W^{k,i},\\
		u(0, \cdot) &= u_0,
	\end{aligned}
\end{equation}
where we denote $\Lambda^{2\alpha} := (-\Delta)^{\alpha}$  and $g_R(u) := g_R\left(\|u\|_{H^{-\delta}}\right), \delta >0$, for a Lipschitz-continuous cut-off function $g_R:[0,\infty) \rightarrow [0,1]$ with $g_R(x) =1$ for $x \in [0,R]$, and $g_R(x) =0$ for $x >R+1$. Similarly to \cite{flandoli2021a} we obtain the corresponding It\^{o}-formulation
\begin{equation}\label{eq:ItoDivfree}
	\begin{aligned}
		{\rm d}u &= \left( \Lambda^{2\alpha}u + g_R(u)F(u)+S_{\theta}(u)\right){\rm d}t \\
		&\hspace{0.5cm}+ \frac{\sqrt{C_d \nu}}{\|\theta\|_{\ell^2}}\sum_{k \in \mathbb{Z}_0^d}\sum_{i=1}^{d-1}\theta_k \Pi((\sigma_{k,i}\cdot\nabla)u) {\rm d}W^{k,i},\\
		u(0, \cdot) &= u_0,
	\end{aligned}
\end{equation}
where
\begin{equation}
	S_{\theta}(u) := \frac{C_d \nu}{\|\theta\|^2_{\ell^2}} \sum_{k \in \mathbb{Z}_0^d}\sum_{i=1}^{d-1}\theta_k^2 \Pi\left((\sigma_{k,i}\cdot\nabla)\Pi((\sigma_{-k,i}\cdot\nabla)u)\right).
\end{equation}
\begin{definition}[cf. {\cite[Definition 1.1]{flandoli2021a}} and {\cite[Definition 3.1]{flandoli2021b}}]
	Let $(\Omega, \mathcal{F}, (\mathcal{F}_t), \mathbb{P})$ be a probability space with a family of Brownian motions $\left\{ W^{k,i} \right\}$ as in Section \ref{subsec:noise}. Further let $u_0 \in L^2(\mathbb{T}^d)$ be divergence-free. A process $u$ with trajectories in $C([0,T],L^2_{\rm df}(\mathbb{T}^d) \cap L^2(0,T;H^{\alpha}_{\rm df}(\mathbb{T}^d))$ is a strong solution to \eqref{eq:StratDivfree}, if it is $(\mathcal{F}_t)$-adapted and for any $\phi \in H^{\alpha}_{\rm df}(\mathbb{T}^d)$ $(\langle u(t), \Pi((\sigma_{k,i}\cdot\nabla)\phi)\rangle_{L^2(\mathbb{T}^d)})_t$ is an $(\mathcal{F}_t)$-continuous semimartingale, and with probability one it holds for all $t \in [0,T]$
	\begin{equation}
		\begin{aligned}
			&\langle u(t), \phi \rangle_{L^2(\mathbb{T}^d)} \\
			&= \langle u_0, \phi \rangle_{L^2(\mathbb{T}^d)} \\
			&\hspace{0,5cm}+ \int_0^t \langle -\Lambda^{2\alpha}u(s)+ g_R(u(s)) F(u(s)) + S_{\theta}(u(s)),\phi\rangle_{L^2(\mathbb{T}^d)}{\rm d}s\\
			&\hspace{0,5cm} + \frac{\sqrt{C_d \nu}}{\|\theta\|_{\ell^2}}\sum_{k \in \mathbb{Z}_0^d}\sum_{i=1}^{d-1}\theta_k \int_0^t\langle u(s), \Pi((\sigma_{k,i}\cdot\nabla)\phi)\rangle_{L^2(\mathbb{T}^d)} {\rm d}W^{k,i}.
		\end{aligned}
	\end{equation}
\end{definition}
In order to show existence of strong solutions in the sense of the above definition, we follow the standard agenda:
\begin{enumerate}
	\item Show existence of global weak solutions to 
	 \begin{equation}\label{eq:ItoCutoff}
	 	\begin{aligned}
	 	{\rm d}u &= \left(-\Lambda^{2\alpha} u + g_R(u)F(u) + S_{\theta}(u)\right){\rm d}t \\
		&\hspace{0.5cm}+\frac{\sqrt{C_d \nu}}{\|\theta\|_{\ell^2}}\sum_{k \in \mathbb{Z}_0^d}\sum_{i=1}^{d-1}\theta_k \Pi((\sigma_{k,i}\cdot\nabla)u) {\rm d}W^{k,i}.
		\end{aligned}
	 \end{equation}
	 \item Show pathwise uniqueness of weak solutions to \eqref{eq:ItoCutoff}.
	 \item Conclude via a Yamada-Watanabe argument.
\end{enumerate}
On (1): For $N \in \mathbb{N}$, let $H_N :=\{ \sigma_{k,i}: |k|\leq N, i=1,...,d-1\}$, and for $H$ the real subspace of $H_{\mathbb{C}}$ (defined in \eqref{eq:HC}) denote the corresponding orthogonal projection by $\Pi_N : H \rightarrow H_N$. Consider the Galerkin approximation of \eqref{eq:ItoCutoff}:
\begin{equation}\label{eq:ItoCutoffGalerkin}
	\begin{aligned}
	 	{\rm d}u_N &= \left(-\Lambda^{2\alpha} u_N + g_R(u_N)\Pi_NF(u_N) + S_{\theta}(u_N)\right){\rm d}t \\
		&\hspace{0.5cm}+\frac{\sqrt{C_d \nu}}{\|\theta\|_{\ell^2}}\sum_{k \in \mathbb{Z}_0^d}\sum_{i=1}^{d-1}\theta_k \Pi_N((\sigma_{k,i}\cdot\nabla)u_N) {\rm d}W^{k,i}.
	\end{aligned}
\end{equation}
We obtain the following a-priori estimates: via It\^{o}'s formula it holds
\begin{equation}
	\begin{aligned}
		&{\rm d}\|u_N(t)\|^2_{L^2} \\
		&=  2 \left\langle u_N(t), -\Lambda^{2\alpha}u_N(t) + g_R( u_N(t) ) \Pi_N F(u_N(t)) + S_{\theta}(u_N(t))\right\rangle_{L^2} {\rm d}t \\
		&\hspace{0.5cm} + 2 \left\langle u_N, \frac{\sqrt{C_d \nu}}{\|\theta\|_{\ell^2}}\sum_{k \in \mathbb{Z}_0^d}\sum_{i=1}^{d-1}\theta_k \Pi_N((\sigma_{k,i}\cdot\nabla)u_N(t)) {\rm d}W^{k,i}\right\rangle_{L^2}\\
		&\hspace{0.5cm} + \frac{C_d \nu}{\|\theta\|^2_{\ell^2}} \sum_{k \in \mathbb{Z}_0^d}\sum_{i=1}^{d-1}\theta_k^2 \left\|\Pi_N((\sigma_{k,i}\cdot\nabla)u_N(t))\right\|^2_{L^2} {\rm d}t.
	\end{aligned}
\end{equation}
Observe that on the one hand since $\sigma_{k,i}$ are divergence-free, it holds
\[ \langle u_N(t), \Pi_N((\sigma_{k,i}\cdot\nabla)u_N(t))\rangle_{L^2} =0\]
hence the martingale part vanishes, and on the other hand as seen in \cite{flandoli2021a} it holds
\[ 2\langle u_N(t), S_{\theta}(u_N(t))\rangle_{L^2} = - \frac{C_d \nu}{\|\theta\|^2_{\ell^2}} \sum_{k \in \mathbb{Z}_0^d}\sum_{i=1}^{d-1}\theta_k^2 \left\|\Pi((\sigma_{k,i}\cdot\nabla)u_N(t))\right\|^2_{L^2}\]
which in total gives
\begin{equation}
	\frac{{\rm d}}{{\rm d}t}\|u_N(t)\|^2_{L^2} \leq -2 \left\| \Lambda^{\alpha}u\right\|_{L^2} + 2\langle u_N(t), g_R(u_N(t)) \Pi_N F(u_N(t))\rangle_{L^2}.
\end{equation}
Using Remark \ref{rem:hypFGL21} (3) (for which we assume that $\delta >0$ is small enough such that (H2') holds) and the analysis in \cite{flandoli2021b} it holds
\begin{equation}
	\begin{aligned}
		&|\langle u_N(t), g_R(u_N(t)) \Pi_N F(u_N(t))\rangle_{L^2}| \\
		&\lesssim 1 + \frac{1}{2^{\alpha}} \|u_N(t)\|_{H^{\alpha}}^2\\
		& \leq 1 + \frac{1}{2}\left(\|u_N(t)\|_{L^2}^2 + \left\|\Lambda^{\alpha}u_N(t)\right\|_{L^2}^2\right),
	\end{aligned}
\end{equation}
hence we obtain a constant $C$ such that 
\begin{equation}\label{eq:uNbound}
	\sup_{t\in[0,T]} \|u_N(t)\|^2_{L^2} + \int_0^T \left\|\Lambda^{\alpha}u_N(t)\right\|^2_{L^2} {\rm d}t \leq C\left(1+\|u_0\|^2_{L^2}\right).
\end{equation}
In the following we consider the case $\alpha >1$ and follow the analysis in \cite{flandoli2021b} (for $\alpha =1$, proceed by using the results in \cite{flandoli2021a}): recall
\begin{lemma}[cf. {\cite[Lemma 3.3]{flandoli2021b}}]\label{lem:compactLemmaFGL}
	For any $\beta, \gamma, \epsilon >0$ and any $p < \infty$ define
	\begin{align*}
		\mathcal{S}^{\gamma, \beta} &:= L^2(0,T;H^{\alpha}(\mathbb{T}^d)) \cap C([0,T];L^2(\mathbb{T}^d)) \cap C^{\gamma}([0,T];H^{\beta}(\mathbb{T}^d)),\\
		\mathcal{X}^{\epsilon, p} &:= L^2(0,T;H^{\alpha - \epsilon}(\mathbb{T}^d)) \cap L^p(0,T;L^2(\mathbb{T}^d)) \cap C([0,T]; H^{-\epsilon}(\mathbb{T}^d)),
	\end{align*}
	then the embedding $\mathcal{S}^{\gamma, \beta} \hookrightarrow \mathcal{X}^{\epsilon, p}$ is compact and for any finite $K\geq 0$ the set
	\[\mathcal{X}_K := \left\{ f \in \mathcal{X}^{\epsilon, p}: \sup_{t \in [0,T]} \|f(t)\|_{L^2(\mathbb{T}^d)} + \|f\|_{L^2H^{\alpha}(\mathbb{T}^d)} \leq K\right\}\]
	is closed in $\mathcal{X}^{\epsilon, p}$ and hence a Polish space with metric inherited from $\mathcal{X}^{\epsilon, p}$.
\end{lemma}
Hence we proceed to verify that there exist $p>1, \beta, \gamma >0$ such that
\begin{equation}\label{eq:compactBound}
	\sup_{N \in \mathbb{N}} \mathbb{E}\left[ \left(\|u_N\|_{L^2H^{\alpha}} + \|u_N\|_{L^{\infty}L^2} + \|u_N\|_{C^{\gamma}H^{-\beta}}\right)^p\right] < \infty.
\end{equation}
Denote
\[M_N(t) := \frac{\sqrt{C_d \nu}}{\|\theta\|_{\ell^2}} \int_0^t \sum_{k \in \mathbb{Z}_0^d}\sum_{i=1}^{d-1}\theta_k \Pi_N((\sigma_{k,i}\cdot\nabla)u_N(s)) {\rm d}W^{k,i}_s,\]
then for any $t,s \in [0,T]$ we deduce from a similar analysis as in \cite{flandoli2021b} that
\begin{equation}
	\mathbb{E}\left[ \|M_N(t) - M_N(s)\|_{H^{-\beta}}^{2p}\right] \lesssim \|\theta\|_{\ell^{\infty}}^{2p} \left(1+\|u_0\|_{L^2}^2\right)^p |t-s|^p.
\end{equation}
Further, by assumption on $F$ it holds for
\[ v_N(t) := \int_0^t -\Lambda^{2\alpha} u_N(s) + g_R(u_N(s))\Pi_N F(u_N(s)) {\rm d}s\]
that
\begin{equation}
	\begin{aligned}
		\|v_N\|_{C^{\frac{1}{2}}H^{-\alpha}} &\leq \|v_N\|_{W^{1,2}H^{-\alpha}} \\
		&\leq \left(1 + \|u_N\|_{L^{\infty}L^2}^{\beta_1}\right)\left(1 + \|u_N\|_{L^2H^{\alpha}}\right).
	\end{aligned}
\end{equation}
It thus remains to estimate
\[\left\| \int_0^{\cdot} S_{\theta}(u_N(s)){\rm d}s\right\|_{W^{1,2}H^{-\alpha}}:\]
Observe from \cite{flandoli2021a} that it holds for any $t,s \in [0,T]$ and $l \in \mathbb{Z}_0^d, j \in \{1,...,d-1\}$
\begin{equation}
	\begin{aligned}
		\left| \left\langle \int_s^t S_{\theta}(u_N(r)) {\rm d}r , \sigma_{l,j}\right\rangle_{L^2}\right| &= \left|  \int_s^t \langle u_N(r)) , S_{\theta}(\sigma_{l,j})\rangle_{L^2}{\rm d}r \right|\\
		&\lesssim \|u_N\|_{L^{\infty}L^2} |l|^2|t-s|
	\end{aligned}
\end{equation}
and hence
\begin{equation}
	\begin{aligned}
		\left\|\int_s^t S_{\theta}(u_N(r)) {\rm d}r \right\|_{H^{-\alpha}}^2 &= \sum_{l,j} \frac{\left| \left\langle \int_s^t S_{\theta}(u_N(r)) {\rm d}r , \sigma_{l,j}\right\rangle_{L^2}\right|^2}{|l|^{-2\alpha}}\\
		& \lesssim \|u_N\|_{L^{\infty}L^2}^2 |t-s|^2\sum_l |l|^{2(1-\alpha)}
	\end{aligned}
\end{equation}
where the series converges for $\alpha >1$. In total this gives
\begin{equation}
	\left\| \int_0^{\cdot} S_{\theta}(u_N(s)){\rm d}s\right\|_{C^{\frac{1}{2}}H^{-\alpha}} \lesssim \|u_N\|_{L^{\infty}L^2}
\end{equation}
and concludes the proof of \eqref{eq:compactBound}. Thus we obtain existence of weak solutions as in \cite{flandoli2021b} (for the reader's convenience we repeat here the core arguments): denote $\mu_N := {\rm Law}(u_N)$, then by Lemma \ref{lem:compactLemmaFGL} and Prokhorov's theorem, the family $\{\mu_N\}_N$ is tight in $\mathcal{X}^{\epsilon, p}$ and we can find $K$ large enough such that $\{\mu_N\}_N$ are supported on $\mathcal{X}_K$ and thus tight therein as well. Existence of weak solutions then follows from Skorokhod's representation theorem: for $W:=(W^{k,i})_{k,i}$ denote $P_N := {\rm Law}(u_N, W)$, then $\{P_N\}_N$ is tight in $\mathcal{X}_K \times C\left([0,T], \mathbb{C}^{\mathbb{Z}_0^2}\right)$. Hence we can find another probability space $(\tilde{\Omega}, \tilde{\mathcal{F}}, (\tilde{\mathcal{F}}_t), \tilde{\mathbb{P}})$ and corresponding $(\tilde{u}_N, \tilde{W}_N)$ such that $(\tilde{u}_N, \tilde{W}_N)$ converge to $(\tilde{u}, \tilde{W})$ $\tilde{\mathbb{P}}$-almost surely in $\mathcal{X}_K \times C\left([0,T], \mathbb{C}^{\mathbb{Z}_0^2}\right)$ and ${\rm Law}(\tilde{u}_N, \tilde{W}_N) = P_N$. Furthermore $\tilde{u}$ solves \eqref{eq:StratDivfree} with $\tilde{W}$ by $\tilde{\mathbb{P}}$-almost sure convergence where the nonlinear part converges due to continuity of $F$ on $\mathcal{X}_K$ (see Lemma 3.4, \cite{flandoli2021b}).\\
\\
\noindent
On (2): For pathwise uniqueness assume that on some probability space $(\Omega, \mathcal{F}, (\mathcal{F}_t), \mathbb{P})$ there exist two weak solutions $u_1, u_2$ to \eqref{eq:ItoDivfree} with the same $W$ and same initial condition $u_0$. Let $\tilde{u}:=u_1 - u_2$, then for $\phi \in C^{\infty}(\mathbb{T}^d)$ it holds
\begin{equation}
	\begin{aligned}
		&\langle \tilde{u}(t), \phi\rangle_{L^2} \\
		&= \int_0^t \langle - \Lambda^{2\alpha}\tilde{u}(s) +\left(g_R(u_1(s))F(u_1(s))-g_R(u_2(s))F(u_2(s))\right) \\
		&\hspace{3cm}+ S_{\theta}(\tilde{u}(s)), \phi\rangle_{L^2} {\rm d}s\\
		&\hspace{0.5cm} + \frac{\sqrt{C_d \nu}}{\|\theta\|_{\ell^2}}  \sum_{k \in \mathbb{Z}_0^d}\sum_{i=1}^{d-1}\theta_k \int_0^t\langle \Pi((\sigma_{k,i}\cdot\nabla)\tilde{u}(s)), \phi\rangle_{L^2} {\rm d}W^{k,i}_s.
	\end{aligned}
\end{equation}
By It\^{o}'s formula we obtain as before
\begin{equation}
	\begin{aligned}
		&\frac{{\rm d}}{{\rm d}t} \|\tilde{u}(t)\|_{L^2}^2 \\
		&= -2\|\Lambda^{\alpha}\tilde{u}(t)\|_{L^2}^2 + \langle g_R(u_1(t))F(u_1(t))-g_R(u_2(t))F(u_2(t)), \tilde{u}(t)\rangle_{L^2}\\
		&= (3.11) \text{ in \cite{flandoli2021b}.}
	\end{aligned}
\end{equation}
Estimating the nonlinear part by the same analysis as in \cite{flandoli2021b} yields
\begin{equation}
	\begin{aligned}
		\frac{{\rm d}}{{\rm d}t} \|\tilde{u}(t)\|_{L^2}^2 &\leq \left(1+ \|u_1(t)\|_{H^{\alpha}}^2 + \|u_2(t)\|_{H^{\alpha}}^2\right) \|\tilde{u}(t)\|_{L^2}^2, \\
		\|\tilde{u}(0)\|_{L^2}^2&=0,
	\end{aligned}
\end{equation}
where the term in brackets is integrable, and hence implies pathwise uniqueness by Gronwall's lemma.\\
Part (3) is then a consequence of the Yamada-Watanabe theorem (e.g. \cite[Theorem 2.1]{roeck2008}).\\
\noindent
Next consider the following choice
\begin{equation}\label{eq:specialTheta}
\theta_k^N = \frac{1}{|k|^{\lambda}}\mathds{1}_{\{N \leq |k| \leq 2N\}}, \quad k \in \mathbb{Z}_0^d, N \in \mathbb{N},
\end{equation}
for some $\lambda >0$. For any smooth divergence-free vector field $\phi$, Theorem 5.1 in \cite{flandoli2021a} implies that
\begin{equation}\label{eq:convStheta}
	\lim_{N \rightarrow \infty} S_{\theta^N}(\phi) = \frac{3\nu}{5}\Delta \phi
\end{equation}
in $L^2(\mathbb{T}^d)$. With this we shall prove the following result:
\begin{theorem}[cf. {\cite[Proposition 3.7]{flandoli2021b}} and {\cite[Theorem 1.4]{flandoli2021a}}]\label{thm:convDet}
	Let $\{u^N_0\}_N\subset L^2_{\rm df}(\mathbb{T}^d)$ converge weakly in $L^2_{\rm df}(\mathbb{T}^d)$ to some $u_0\in L^2_{\rm df}(\mathbb{T}^d)$. Further let $u^N$ denote the unique strong solution to \eqref{eq:StratDivfree} associated to $\theta^N$ defined in \eqref{eq:specialTheta} starting at $u_0^N$. Then for every $\epsilon >0, p \geq 2$, $u^N$ converges in probability in the topology of $\mathcal{X}^{\epsilon, p}$ to the unique solution $u := u(\cdot; u_0, \nu)$ to 
\begin{equation}\label{eq:scalLim}
	\begin{aligned}
		\partial_t u &= -\Lambda^{2\alpha} u + \frac{3\nu}{5}\Delta u + g_R(u)F(u),\\
		u(0,\cdot) &= u_0.
	\end{aligned}
\end{equation}
\end{theorem}
\begin{proof}
Recall from \eqref{eq:uNbound} that it holds
\begin{equation*}
	\sup_{t\in[0,T]} \|u^N(t)\|^2_{L^2} + \int_0^T \left\|\Lambda^{\alpha}u^N(t)\right\|^2_{L^2} {\rm d}t \leq C\left(1+\|u_0^N\|^2_{L^2}\right).
\end{equation*}
Since $(u_0^N)_N$ is weakly convergent in $L^2(\mathbb{T}^d)$, it is bounded therein and we may thus find a constant $C$ such that
\begin{equation}\label{eq:unifBound}
	\sup_{N \in \mathbb{N}} \left(\sup_{t\in[0,T]} \|u^N(t)\|^2_{L^2} + \int_0^T \left\|\Lambda^{\alpha}u^N(t)\right\|^2_{L^2} {\rm d}t\right)\leq C \quad \mathbb{P}\text{-a.s.}
\end{equation}
Similar to previous estimates we may further find $q>1, \beta, \gamma >0$ such that
\begin{equation}
	\sup_{N \in \mathbb{N}} \mathbb{E}\left[ \left( \|u^N\|_{L^{\infty}L^2} + \|u^N\|_{L^2H^{\alpha}} + \|u^N\|_{C^{\gamma}H^{-\beta}}\right)^q\right] < \infty.
\end{equation}
Thus for every $\epsilon >0$ and $q \geq 2$ by \cite[Lemma 3.3]{flandoli2021b}, the sequence of laws $(\mu^N)_N, \mu^N := {\rm Law}(u^N)$, is tight in $\mathcal{X}^{\epsilon, q}$ and, for $K$ big enough, also in $\mathcal{X}_K$ (by \eqref{eq:unifBound}). By Prokhorov's theorem, we may thus extract a subsequence $(\mu^{N_n})_n$ which converges weakly to some probability measure $\mu$ on $\mathcal{X}_K$. For every $\phi \in C^{\infty}(\mathbb{T}^d)$, let the function $T^{\phi}:\mathcal{X}_K \rightarrow C([0,T];\mathbb{R})$ be defined by
\begin{align*}
\left(T^{\phi}f\right)(t)&:= \langle f(t), \phi \rangle_{L^2} - \langle u_0, \phi \rangle_{L^2} \\
&\hspace{0.5cm}- \int_0^t g_R(f(s)) \langle F(f(s)), \phi \rangle_{L^2} {\rm d}s\\
&\hspace{0.5cm} - \int_0^t \langle f(s),- \Lambda^{2\alpha} \phi + S_{\theta^N}(\phi)\rangle_{L^2}{\rm d}s,
\end{align*}
then due to Lemma 3.4 in \cite{flandoli2021b} and \eqref{eq:convStheta}, $T^{\phi}$ is continuous on $\mathcal{X}_K$. Following the same reasoning as in \cite{flandoli2021b}, it therefore holds $\mu=\delta_u$ where $u$ is the unique solution to \eqref{eq:scalLim} which concludes the proof.
\end{proof}
We conclude as in the proof of \cite[Theorem 1.4]{flandoli2021b}: Let $\epsilon>0 ,T\in (0,\infty)$ be fixed, then by hypothesis (H4) there exist $\nu>0, R >1$ such that
\[ \sup_{u_0 \in \mathcal{K}\cap \mathcal{D}} \sup_{t \in [0,T]} \|u(t; u_0, \nu)\|_{L^2} \leq R-1\]
hence $u(\cdot;u_0, \nu)$ is a solution to the deterministic equation without cut-off. Let $u^R(\cdot; u_0, \theta^N, \nu)$ denote the solution to the cut-off equation \eqref{eq:StratDivfree} with $\theta^N$ as in \eqref{eq:specialTheta}. By choice of $\mathcal{K}$ and due to Theorem \ref{thm:convDet} it holds
\begin{equation}
	\begin{aligned}
		&\lim_{N \rightarrow \infty} \sup_{u_0 \in \mathcal{K}\cap \mathcal{D}} \mathbb{P}\left[\sup_{t \in [0,T]} \left\|u^R(t;u_0,\theta^N,\nu)\right\|_{H^{-\delta}}>R\right]\\
		&\leq \lim_{N \rightarrow \infty} \sup_{u_0 \in \mathcal{K}\cap \mathcal{D}} \mathbb{P}\left[\sup_{t \in [0,T]} \left\|u^R(t;u_0,\theta^N,\nu)-u(t;u_0,\nu)\right\|_{H^{-\delta}}>1\right]=0
	\end{aligned}
\end{equation}
and hence there exists $N$ big enough such that uniformly in $u_0 \in \mathcal{K}\cap \mathcal{D}$ it holds
\begin{equation}
	\mathbb{P}\left[\sup_{t \in [0,T]} \left\|u^R(t;u_0,\theta^N,\nu)\right\|_{H^{-\delta}}\leq R\right]>1 - \epsilon,
\end{equation}
i.e. $u^R$ solves \eqref{eq:StratDivfree} without cut-off.

\end{document}